\documentclass[11pt]{amsart}

\usepackage[margin=1.2in]{geometry}
\usepackage{amsmath, amssymb, amsthm, mathtools}
\usepackage{tikz}
\usetikzlibrary{arrows.meta, decorations.markings, calc, positioning}
\usepackage{aliascnt}
\usepackage[colorlinks=true, linkcolor=blue!60!black, citecolor=blue!60!black]{hyperref}
\usepackage[capitalize]{cleveref}

\theoremstyle{plain}
\newtheorem{theorem}{Theorem}[section]
\newaliascnt{lemma}{theorem}
\newtheorem{lemma}[lemma]{Lemma}
\aliascntresetthe{lemma}
\newaliascnt{proposition}{theorem}
\newtheorem{proposition}[proposition]{Proposition}
\aliascntresetthe{proposition}
\newaliascnt{corollary}{theorem}
\newtheorem{corollary}[corollary]{Corollary}
\aliascntresetthe{corollary}
\newaliascnt{claim}{theorem}

\aliascntresetthe{claim}
\newtheorem*{mainthm}{Main Theorem}

\theoremstyle{definition}
\newaliascnt{definition}{theorem}
\newtheorem{definition}[definition]{Definition}
\aliascntresetthe{definition}
\newaliascnt{convention}{theorem}
\newtheorem{convention}[convention]{Convention}
\aliascntresetthe{convention}
\newaliascnt{example}{theorem}

\aliascntresetthe{example}

\theoremstyle{remark}
\newaliascnt{remark}{theorem}
\newtheorem{remark}[remark]{Remark}
\aliascntresetthe{remark}

\newcommand{\supp}{\operatorname{supt}}
\newcommand{\spn}{\operatorname{span}}
\DeclareMathOperator{\Cl}{Cl}
\newcommand{\GB}{\mathcal{G}}               
\newcommand{\Kdyn}{K_{\mathrm{dyn}}}
\newcommand{\io}{\iota}
\newcommand{\ta}{\tau}
\newcommand{\eL}{e_{L}}
\newcommand{\eR}{e_{R}}
\newcommand{\DG}{\mathcal{D}}               
\newcommand{\Pres}{\mathcal{P}}
\newcommand{\R}{\mathbb{R}}
\newcommand{\Z}{\mathbb{Z}}
\newcommand{\Vin}{V_{\mathrm{in}}}
\newcommand{\gap}[1]{\Gamma_{#1}}           
\newcommand{\toppath}[1]{\lceil #1\rceil}
\newcommand{\botpath}[1]{\lfloor #1\rfloor}

\begin{document}

\title[Irreducible fast bump groups are Thompson's groups]
      {Irreducible fast sets of bump homeomorphisms generate copies of Thompson's groups $F_n$}

\author{Gili Golan}
\address{Department of Mathematics, Ben Gurion University of the Negev}
\email{golangi@bgu.ac.il}
\thanks{This research was supported by the ISRAEL SCIENCE FOUNDATION (grant 2275/24).}

\begin{abstract}
A homeomorphism of an interval is a \emph{positive bump} if its support
is a single open interval on which it moves every point to the right.
Choosing a fundamental domain $[m,b(m))$ for the action of a positive
bump $b$ on its support splits the remainder of the support into two
intervals, called the \emph{feet} of $b$.  A finite set of positive
bumps is \emph{geometrically fast} if fundamental domains can be chosen
so that all the resulting feet are pairwise disjoint.  The
\emph{crossing graph} of such a set has the bumps as its vertices, two
bumps being adjacent whenever their supports overlap but are not
nested, and the set is \emph{irreducible} if its crossing graph is
connected.  We prove that for every $n\geq2$, every group generated by
an irreducible geometrically fast set of $n$ positive bumps is
isomorphic to the $n$-ary Thompson group $F_n$.  This answers the
strong version of a problem posed by Brin and Zaremsky in the
Oberwolfach report ``Cohomological and metric properties of groups of
homeomorphisms of $\mathbb{R}$'' (Oberwolfach Rep.\ \textbf{15} (2018)).
\end{abstract}

\maketitle

\section{Introduction}\label{sec:intro}

Let $I\subseteq\R$ be an interval.  A \emph{positive bump} is an
orientation-preserving homeomorphism $b$ of $I$ whose support is a
single open interval
\[
\supp(b)=(a_b,c_b)
\]
and which moves every point of its support to the right.  Choose a
point $m_b\in\supp(b)$, called a \emph{marker}.  The interval
$[m_b,b(m_b))$ is a fundamental domain for the action of the cyclic
group $\langle b\rangle$ on $\supp(b)$, and its complement in the
support consists of two intervals,
\[
L_b=(a_b,m_b)
\qquad\text{and}\qquad
R_b=[b(m_b),c_b),
\]
called the left and right \emph{feet} of $b$.  Following Bleak,
Brin, Kassabov, Moore and Zaremsky \cite{BBKMZ}, a finite set $B$ of
positive bumps is \emph{geometrically fast}, or simply \emph{fast}, if
markers can be chosen so that all the feet of all the bumps in $B$ are
pairwise disjoint.

\begin{figure}[t]
\centering
\begin{tikzpicture}[scale=1.0, every node/.style={font=\small}]
\begin{scope}
    \draw[->] (-0.3,0) -- (3.6,0);
    \draw[line width=2.4pt,blue!65!black] (0,0) -- (0.55,0);
    \draw[line width=2.4pt,blue!65!black] (0.95,0) -- (1.5,0);
    \draw[thick,blue!65!black] (0,0) .. controls (0.35,0.85) and (1.15,0.85) .. (1.5,0);
    \draw[line width=2.4pt,red!70!black] (1.85,0) -- (2.4,0);
    \draw[line width=2.4pt,red!70!black] (2.8,0) -- (3.35,0);
    \draw[thick,red!70!black] (1.85,0) .. controls (2.2,0.85) and (3.0,0.85) .. (3.35,0);
    \node at (1.65,-0.5) {$L_1<R_1<L_2<R_2$};
    \node at (1.65,-1.05) {$\Z\times\Z$};
  \end{scope}
  \begin{scope}[xshift=4.6cm]
    \draw[->] (-0.3,0) -- (3.6,0);
    \draw[line width=2.4pt,blue!65!black] (0,0) -- (0.5,0);
    \draw[line width=2.4pt,blue!65!black] (2.85,0) -- (3.35,0);
    \draw[thick,blue!65!black] (0,0) .. controls (0.55,1.55) and (2.8,1.55) .. (3.35,0);
    \draw[line width=2.4pt,red!70!black] (1.05,0) -- (1.55,0);
    \draw[line width=2.4pt,red!70!black] (1.8,0) -- (2.3,0);
    \draw[thick,red!70!black] (1.05,0) .. controls (1.35,0.8) and (2.0,0.8) .. (2.3,0);
    \node at (1.65,-0.5) {$L_1<L_2<R_2<R_1$};
    \node at (1.65,-1.05) {$\Z\wr\Z$};
  \end{scope}
  \begin{scope}[xshift=9.2cm]
    \draw[->] (-0.3,0) -- (3.6,0);
    \draw[line width=2.4pt,blue!65!black] (0,0) -- (0.55,0);
    \draw[line width=2.4pt,blue!65!black] (1.5,0) -- (2.05,0);
    \draw[thick,blue!65!black] (0,0) .. controls (0.5,1.1) and (1.55,1.1) .. (2.05,0);
    \draw[line width=2.4pt,red!70!black] (0.85,0) -- (1.4,0);
    \draw[line width=2.4pt,red!70!black] (2.35,0) -- (2.9,0);
    \draw[thick,red!70!black] (0.85,0) .. controls (1.35,1.1) and (2.4,1.1) .. (2.9,0);
    \node at (1.65,-0.5) {$L_1<L_2<R_1<R_2$};
    \node at (1.65,-1.05) {$F$};
  \end{scope}
\end{tikzpicture}
\caption{The three
possible dynamical diagrams of two fast bumps, up to reflection:
disjoint, nested and crossing, together with the groups their fast
realizations generate \cite{BBKMZ}.}
\label{fig:intro-fast}
\end{figure}
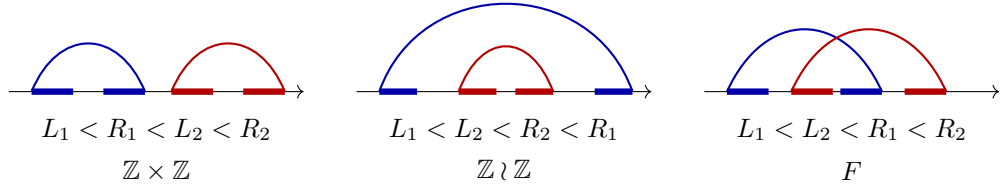

Since the feet of a fast set are pairwise disjoint intervals, they
inherit a left-to-right order.  The \emph{dynamical diagram} of a fast
set records this order, together with the pairing of the two feet
belonging to each bump.  One may draw the diagram by placing the feet on
a line and joining the two feet of each bump by an arc, as in
\cref{fig:intro-fast}.  A fundamental theorem of \cite{BBKMZ} says that
this finite combinatorial object determines the marked isomorphism
type: if two fast sets have the same dynamical diagram, then the natural
bijection between their bumps extends to an isomorphism between the
groups they generate (see \cref{thm:diagram-isomorphism}).

Two bumps $x$ and $y$ of a fast set \emph{cross} when their feet
interlace:
\[
L_x<L_y<R_x<R_y
\qquad\text{or}\qquad
L_y<L_x<R_y<R_x;
\]
equivalently, when their supports overlap and neither contains the
other.  The \emph{crossing graph} of a fast set has one vertex for each
bump and one edge for each crossing pair, and, following Belk and Stott
\cite{BS}, we say that a fast set is \emph{irreducible} when its crossing graph is
connected.  As noted in \cite{BBN}, a fast set of positive bumps that is
not irreducible generates a group which decomposes non-trivially as a
direct product or as a permutational wreath product of groups generated
by smaller fast sets.  For $n\geq2$, let $\mathcal C_n$ denote the class
of groups generated by irreducible geometrically fast sets of $n$
positive bumps (\cref{def:Cn}).

The basic example of an irreducible fast set is a \emph{chain of $n$
bumps}: a fast set $\{b_1,\dots,b_n\}$ whose feet are ordered
\[
L_1<L_2<R_1<L_3<R_2<\cdots<L_n<R_{n-1}<R_n,
\]
so that consecutive bumps cross and the supports of non-consecutive
bumps are disjoint.  Its crossing graph is a path, and the group it
generates is the $n$-ary Thompson group $F_n$ \cite{BBKMZ}: the group of
piecewise-linear orientation-preserving homeomorphisms of $[0,1]$ with
finitely many breakpoints, all in $\Z[1/n]$, and slopes integral powers
of $n$, so that $F_2$ is Thompson's group $F$.  In particular, $F_n$ belongs to $\mathcal C_n$.

The problem of identifying the groups in $\mathcal C_n$ was raised in
the Oberwolfach report \cite{BBN}.  Questions~71 and~72 of the report
ask, for all $n\geq2$, whether every group in $\mathcal C_n$ is of
type $F_\infty$ and whether all the groups in $\mathcal C_n$ are
isomorphic; equivalently, whether they are all isomorphic to
Thompson's group $F_n$.  A similar question appears there as
Question~109, attributed to Brin and Zaremsky, and possibly others.
The report records that the answer to the isomorphism question was
suspected to be negative for $n>3$.

The answer was previously known to be positive in small cases.  For
$n=2$ it is immediate: the only irreducible dynamical diagram of two
bumps is the crossing pair, that is, the chain of two bumps
(\cref{fig:intro-fast}), so every group in $\mathcal C_2$ is isomorphic
to $F$ \cite{BBKMZ}.  For $n=3$, Brin, Bleak and Moore
showed that every group in $\mathcal C_3$ is isomorphic to the $3$-ary
Thompson group $F_3$ \cite[p.~1610]{BBN}: using conjugations that
preserve both fastness and the generated group, every irreducible
dynamical diagram of three bumps can be transformed into the chain of
three bumps.  For $n=4$, they showed that the groups in $\mathcal C_4$
form at most two isomorphism classes, one of which is the $4$-ary
Thompson group $F_4$ and the other of which they named pseudo-$F_4$;
however, they were not able to determine whether pseudo-$F_4$ is
isomorphic to $F_4$ \cite{BBN}.  Belk and Stott subsequently proved
that pseudo-$F_4$ is isomorphic to $F_4$ \cite{BS}; hence every group
in $\mathcal C_4$ is isomorphic to $F_4$.

Our main theorem shows that, contrary to what was suspected, this
pattern persists for every $n$, and it determines the classes
$\mathcal C_n$ completely.

\begin{mainthm}
For every $n\geq2$, every group generated by an irreducible
geometrically fast set of $n$ positive bumps is isomorphic to Thompson's group $F_n$.  Equivalently, $\mathcal C_n$ consists of
a single isomorphism class, represented by $F_n$.
\end{mainthm}

We conclude the introduction with a brief sketch of the proof, which
combines the dynamics of groups of homeomorphisms of the interval with
the theory of diagram groups; it has two halves.  The first is
dynamical: we introduce the \emph{swap move}, a local rearrangement of
a dynamical diagram which is implemented on any fast realization by
conjugating some of the generators by one of the bumps, so that it
preserves fastness and does not change the generated group.  Using swap
moves, we show that every irreducible fast set can be replaced, without
changing the group it generates, by one whose dynamical diagram is
\emph{peelable}: it remains irreducible under repeated deletion of the
rightmost bump.  This is what makes an induction on the number of
bumps possible.

The second half takes place in the world of diagram groups.  By a
theorem of Belk and Stott \cite{BS}, the group generated by a fast set
is a diagram group over an explicit semigroup presentation read off
from its dynamical diagram.  We realize (an extension of) this
presentation as a directed $2$-complex and show, by induction along the
peeling and using the Tietze-like moves of Guba and Sapir \cite{GS06},
that for a peelable fast set of $n$ bumps the complex can be
transformed, without changing its diagram group, into one of a family
of model complexes built around the Stallings core of $F_n$
\cite{GolanSapirTAMS,GGSFn}, which we call \emph{dilated cores} and
whose diagram groups (over certain base paths) we prove to be
isomorphic to $F_n$.

\subsection*{Organization of the paper}
\Cref{sec:prelim} collects the necessary background on fast sets and
dynamical diagrams, on diagram groups over directed $2$-complexes and
the Guba--Sapir moves, and on the groups $F_n$.  \Cref{sec:swap}
introduces the swap move and proves the peelability theorem.
\Cref{sec:cores} presents the core of $F_n$ and introduces dilated
cores.  \Cref{sec:complex} constructs the directed $2$-complex of a
dynamical diagram.  \Cref{sec:reduction} proves the Reduction Theorem,
and \cref{sec:mainproof} completes the proof of the Main Theorem.

\subsection*{Acknowledgments}
The author would like to thank Eytan Sapir for helpful conversations
regarding the paper \cite{BS}.

\section{Preliminaries}\label{sec:prelim}

\subsection{Bumps, fast sets and dynamical diagrams}\label{subsec:fast}

Let $\operatorname{Homeo}^+(I)$ denote the group of orientation-preserving
homeomorphisms of an interval $I\subseteq\R$.  For
$f\in\operatorname{Homeo}^+(I)$ the \emph{support} of $f$ is
$\supp(f)=\{x\in I\mid f(x)\neq x\}$, and an \emph{orbital} of $f$ is a
maximal open interval on which $f$ has no fixed point (equivalently, a
connected component of $\supp(f)$).

\begin{definition}[Bumps]\label{def:bump}
A homeomorphism $b\in\operatorname{Homeo}^+(I)$ is a \emph{bump} if it has
exactly one orbital, and a bump $b$ is \emph{positive} if $b(x)\geq x$ for
all $x\in I$.  Thus the support of a positive bump $b$ is a single open
interval $(a_b,c_b)$, and $b(x)>x$ for all $x\in(a_b,c_b)$.
For a set $B$ of homeomorphisms, write
\[
\supp(B):=\supp(\langle B\rangle)
=\bigcup_{b\in B}\supp(b).
\]
\end{definition}

\begin{definition}[Markings, feet and fast sets; {\cite[\S3]{BBKMZ}}]\label{def:fast}
Let $B$ be a finite set of positive bumps, and write
$\supp(b)=(a_b,c_b)$.  A \emph{marking} of $B$ is a choice of a point
$m_b\in\supp(b)$ (a \emph{marker}) for each $b\in B$.  The marker
determines the two \emph{feet} of $b$: the \emph{left foot}
\[
L_b=(a_b,\,m_b)
\]
and the \emph{right foot}
\[
R_b=[\,b(m_b),\,c_b),
\]
also called the \emph{source} and the \emph{destination} of $b$.  The
support of $b$ is the disjoint union
\[
\supp(b)=L_b\sqcup[\,m_b,b(m_b))\sqcup R_b,
\]
and the fundamental domain $[\,m_b,b(m_b))$ is called the \emph{gap} of
$b$.

The set $B$ is \emph{geometrically fast} (or simply \emph{fast}) if it
admits a marking for which the $2|B|$ feet are pairwise disjoint.  Such a
marking is said to \emph{witness} fastness, and one is fixed whenever we
speak of a marked fast set.
\end{definition}

Since the feet of a fast set are pairwise disjoint intervals, they carry a
strict left-to-right order on the line.  The dynamical diagram records this
order, together with the information of which bump owns each foot and
whether it is a left or a right foot.

\begin{definition}[The dynamical diagram of a fast set]\label{def:diagram-fast}
Let $B=\{b_1,\dots,b_n\}$ be a fast set of positive bumps with the
indicated enumeration and with a marking witnessing fastness.  The
\emph{dynamical diagram of $B$} is the linear order on the formal letters
$L_1,R_1,\dots,L_n,R_n$ in which $L_i$ and $R_i$ occupy the positions of
the left and right feet of $b_i$, respectively.
\end{definition}

\begin{remark}[The diagram does not depend on the marking]\label{rem:marking-indep}
Any two markings witnessing fastness give the same order of the feet, so
the dynamical diagram depends only on the enumerated set of bumps.  The
order is determined by the support endpoints; for distinct $b,b'\in B$,
\[
L_b<L_{b'}\iff a_b<a_{b'},\qquad
R_b<R_{b'}\iff c_b<c_{b'},\qquad
L_b<R_{b'}\iff a_b<c_{b'}.
\]
Moreover, geometric fastness itself is intrinsic to $B$: a criterion in
terms of the support endpoints and their dynamics, together with a
construction of a witnessing marking, is given in
\cite[Proposition~4.3]{BBKMZ}.
\end{remark}

We now regard these linear orders as combinatorial objects in their own
right.

\begin{definition}[Dynamical diagrams]\label{def:dyndiag}
A \emph{dynamical diagram} on $n$ bumps is a linear order $D$, written
$<$, on the $2n$ formal letters $L_1,R_1,\dots,L_n,R_n$ such that
$L_i<R_i$ for every $i$.  The \emph{$i$-th bump} of $D$ is the formal
pair $\{L_i,R_i\}$, whose members are its \emph{left foot} and
\emph{right foot}; its \emph{span} is the open position interval
$\spn(i)=(L_i,R_i)$.

Two distinct bumps cross if their feet interlace, that is, if
\[
L_i<L_j<R_i<R_j
\qquad\text{or}\qquad
L_j<L_i<R_j<R_i.
\]
The $j$-th bump is
\emph{nested} in the $i$-th if
\[
L_i<L_j<R_j<R_i.
\]
If two bumps do not cross and neither is nested in the other, then their
spans are disjoint.
For a subset $S$ of the bumps, the \emph{sub-diagram} $D|_S$ is the
induced order on the feet of the bumps in $S$.

A fast set $B$ of $n$ positive bumps \emph{realizes} $D$ if it admits an
enumeration $B=\{b_1,\dots,b_n\}$ for which the dynamical diagram of
\cref{def:diagram-fast} is $D$.  A marked fast set with this property is
a \emph{fast realization} of $D$.
\end{definition}

In a realization $B=\{b_1,\dots,b_n\}$, the bumps in the diagram
corresponding to $b_i$ and $b_j$ cross precisely when their supports
overlap and neither contains the other.  The bump corresponding to $b_j$
is nested in the bump corresponding to $b_i$ precisely when
$\supp(b_j)\subseteq\supp(b_i)$.  Finally, the spans of the two diagram
bumps are disjoint precisely when the supports of $b_i$ and $b_j$ are
disjoint.

The dynamical diagrams of fast sets of positive bumps are exactly the
linear orders on $L_1,R_1,\dots,L_n,R_n$ satisfying $L_i<R_i$ for every
$i$: necessity is immediate, and sufficiency is
\cite[Theorem~1.2]{BBKMZ}.

We shall use the following isomorphism theorem.

\begin{theorem}[{\cite[Theorem~1.1]{BBKMZ}}]\label{thm:diagram-isomorphism}
If $B=\{b_1,\dots,b_n\}$ and $B'=\{b'_1,\dots,b'_n\}$ are geometrically
fast sets of positive bumps and the bijection $b_i\mapsto b'_i$ identifies
their dynamical diagrams, then this bijection extends to an isomorphism
between the groups they generate.
\end{theorem}

\begin{definition}[Isolated bumps; {\cite[\S3]{BBKMZ}}]\label{def:isolated}
Let $B$ be a fast set.  A bump $b\in B$ is \emph{isolated} if
$\supp(b)$ contains no endpoint of the support of any bump in $B$.
Equivalently, no foot of another bump lies strictly between the two feet
of $b$ in the dynamical diagram.
\end{definition}

\begin{definition}[Canonical configuration; {\cite[\S3.1]{BS}}]
\label{def:canonical-config}
A marked fast set $B$ is in \emph{canonical configuration} if its feet,
together with the gap $[m_b,b(m_b))$ of each isolated bump $b$, cover
$\supp(B)$ up to finitely many points.
\end{definition}

\begin{remark}[Existence of canonical configurations]
\label{rem:canonical-exists}
Every dynamical diagram has a fast realization in canonical
configuration.  Indeed, add inside each isolated bump the two auxiliary
crossing bumps from \cite[Proposition~3.2]{BBKMZ}, and apply the interval
realization in the proof of \cite[Theorem~7.1]{BBKMZ} to the enlarged
diagram.  After discarding the auxiliary bumps, the intervals formerly
occupied by their feet form the single fundamental-domain gap
$[m_b,b(m_b))$ of the corresponding isolated bump $b$, while all
remaining partition intervals are feet.  Thus these intervals cover the
support up to finitely many endpoints.
\end{remark}

\begin{definition}[Crossing graph, irreducibility]\label{def:crossing}
The \emph{crossing graph} $\GB(D)$ of a dynamical diagram $D$ has the
bumps as vertices and the crossing pairs as edges.  The diagram $D$ is
\emph{irreducible} if $\GB(D)$ is connected.  For a fast set $B$ we write
$\GB(B)$ for the crossing graph of its dynamical diagram and call $B$
\emph{irreducible} if $\GB(B)$ is connected.
\end{definition}

The configurations in \cref{fig:dyndiag} illustrate crossing, nesting,
and an isolated bump.

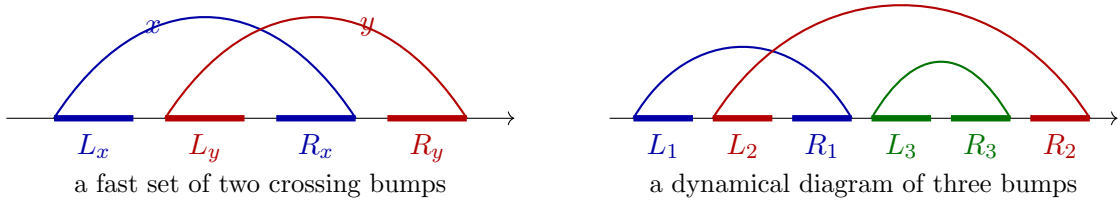
\begin{figure}[ht]
\centering
\begin{tikzpicture}[scale=1.05]
  \begin{scope}
    \draw[->] (-0.3,0) -- (6.1,0);
    \draw[line width=2.4pt,blue!65!black] (0.3,0) -- (1.3,0);
    \node[blue!65!black,below] at (0.8,-0.07) {$L_x$};
    \draw[line width=2.4pt,blue!65!black] (3.1,0) -- (4.1,0);
    \node[blue!65!black,below] at (3.6,-0.07) {$R_x$};
    \draw[thick,blue!65!black] (0.3,0) .. controls (1.4,1.7) and (3.0,1.7) .. (4.1,0);
    \node[blue!65!black] at (1.55,1.15) {$x$};
    \draw[line width=2.4pt,red!70!black] (1.7,0) -- (2.7,0);
    \node[red!70!black,below] at (2.2,-0.07) {$L_y$};
    \draw[line width=2.4pt,red!70!black] (4.5,0) -- (5.5,0);
    \node[red!70!black,below] at (5.0,-0.07) {$R_y$};
    \draw[thick,red!70!black] (1.7,0) .. controls (2.8,1.7) and (4.4,1.7) .. (5.5,0);
    \node[red!70!black] at (4.25,1.15) {$y$};
    \node at (2.9,-0.85) {\small a fast set of two crossing bumps};
  \end{scope}
  \begin{scope}[xshift=7.6cm]
    \draw[->] (-0.3,0) -- (6.1,0);
    \draw[line width=2.4pt,blue!65!black] (0,0) -- (0.75,0);
    \node[blue!65!black,below] at (0.375,-0.07) {$L_1$};
    \draw[line width=2.4pt,blue!65!black] (2,0) -- (2.75,0);
    \node[blue!65!black,below] at (2.375,-0.07) {$R_1$};
    \draw[thick,blue!65!black] (0,0) .. controls (0.7,1.2) and (2.05,1.2) .. (2.75,0);
    \draw[line width=2.4pt,red!70!black] (1,0) -- (1.75,0);
    \node[red!70!black,below] at (1.375,-0.07) {$L_2$};
    \draw[line width=2.4pt,red!70!black] (5,0) -- (5.75,0);
    \node[red!70!black,below] at (5.375,-0.07) {$R_2$};
    \draw[thick,red!70!black] (1,0) .. controls (2.2,1.9) and (4.55,1.9) .. (5.75,0);
    \draw[line width=2.4pt,green!45!black] (3,0) -- (3.75,0);
    \node[green!45!black,below] at (3.375,-0.07) {$L_3$};
    \draw[line width=2.4pt,green!45!black] (4,0) -- (4.75,0);
    \node[green!45!black,below] at (4.375,-0.07) {$R_3$};
    \draw[thick,green!45!black] (3,0) .. controls (3.55,0.95) and (4.2,0.95) .. (4.75,0);
    \node at (2.9,-0.85) {\small a dynamical diagram of three bumps};
  \end{scope}
\end{tikzpicture}
\caption{Feet and bump arcs.  Each bump and its two feet share a color,
and each bump is drawn as an arc joining the left endpoint of its left
foot to the right endpoint of its right foot, so that the arc spans the
support of the bump.  Left: a fast set of two crossing bumps,
$L_x<L_y<R_x<R_y$; the group they generate is Thompson's group $F$
\cite{BBKMZ}.  Right: a dynamical diagram of three bumps in which bumps
$1$ and $2$ cross, bump $3$ is nested in bump $2$, and bumps $1$ and $3$
have disjoint spans.  This diagram is not irreducible (bump $3$ is
isolated), so the groups generated by its fast realizations do not
belong to the class $\mathcal C_3$ of \cref{def:Cn}.}
\label{fig:dyndiag}
\end{figure}

\begin{definition}[The class $\mathcal{C}_n$]\label{def:Cn}
For $n\geq 2$, $\mathcal{C}_n$ denotes the class of groups
$\langle B\rangle$ generated by irreducible geometrically fast sets $B$
of $n$ positive bumps.
\end{definition}

Finally we recall, in an informal form, the representation theorem of Belk
and Stott; the precise statement, in the language of directed
$2$-complexes, is \cref{thm:representation} below.

\begin{theorem}[{\cite[Theorem~1.2]{BS}}]\label{thm:BS-informal}
Let $B$ be a geometrically fast set of positive bumps.  Then
$\langle B\rangle$ is isomorphic to the diagram group, over an explicit
finite semigroup presentation read off from the dynamical diagram of $B$,
with an explicit base word.
\end{theorem}

\subsection{Directed \texorpdfstring{$2$}{2}-complexes and diagram groups}
\label{subsec:dg}

Diagram groups were introduced by Meakin and Sapir; see the historical
account in \cite[Introduction]{GS06}.  They were first studied by
Kilibarda \cite{Kilibarda}, and the theory over semigroup presentations
was developed by Guba and Sapir in \cite{GS97}.  We use the framework of
directed $2$-complexes from \cite{GS06}.  This framework produces no new
groups---the diagram groups of directed $2$-complexes are exactly those
of semigroup presentations \cite[Theorem~4.3]{GS06}---but its vertex
structure is useful throughout the paper.

\begin{definition}[Directed $2$-complex; {\cite[Definition~2.1]{GS06}}]
\label{def:2complex}
A \emph{directed $2$-complex} $K$ consists of the following data.
\begin{enumerate}
\item A directed graph $K^{(1)}$, called the \emph{$1$-skeleton}, with
vertex set $V$, edge set $E$, and initial and terminal maps
\[
\io,\ta\colon E\longrightarrow V.
\]
\item A set $F^+$ of \emph{positive $2$-cells}.  Each $f\in F^+$ has
non-empty directed $1$-paths
\[
\toppath f,\qquad \botpath f
\]
in $K^{(1)}$ with the same initial and terminal vertices.  They are the
\emph{top} and \emph{bottom paths} of $f$; their common initial and
terminal vertices are denoted $\io(f)$ and $\ta(f)$.
\item For every $f\in F^+$ a formal inverse cell $f^{-1}$, with
\[
\toppath{f^{-1}}=\botpath f,\qquad
\botpath{f^{-1}}=\toppath f,
\qquad (f^{-1})^{-1}=f.
\]
\end{enumerate}
We write $F^- =\{f^{-1}\mid f\in F^+\}$ and $F=F^+\sqcup F^-$.  In
specifying a complex we normally list only its positive cells, writing
$f\colon u\to v$, or simply $u=v$ when the cell name is clear, for a
positive cell with top path $u$ and bottom path $v$.

Directed edge-paths in $K^{(1)}$, including the empty path at each
vertex, are called \emph{$1$-paths}; since all paths are directed, we
also call them \emph{positive paths}.  Their initial and terminal
vertices are denoted $\io(p)$ and $\ta(p)$.
\end{definition}

\begin{remark}[Words and $1$-paths]\label{rem:readable}
A word in the alphabet $E$ is a $1$-path if and only if every two
consecutive letters $e,e'$ satisfy $\ta(e)=\io(e')$.  Thus every subword
of a $1$-path is a $1$-path.  Moreover, if $q=a w b$ is a $1$-path and
$w'$ is a $1$-path with the same initial and terminal vertices as $w$,
then $a w'b$ is a $1$-path with the same endpoints as $q$.
\end{remark}

\begin{definition}[Diagrams; {\cite[Definition~2.4]{GS06}}]
\label{def:diagram}
A \emph{diagram} $\Delta$ over $K$ is a finite connected plane directed
$2$-complex, with edges labelled by edges of $K$ and bounded faces
labelled by cells of $K$, satisfying the following conditions.
\begin{enumerate}
\item It has a unique source $\io(\Delta)$ and a unique sink
$\ta(\Delta)$, and every directed path in $\Delta$ is simple.
\item The boundary of the unbounded face is the union of two directed
paths from $\io(\Delta)$ to $\ta(\Delta)$, the \emph{top path}
$\toppath\Delta$ and the \emph{bottom path} $\botpath\Delta$.
\item The boundary of every bounded face $\pi$ is likewise the union of
an upper and a lower directed path from $\io(\pi)$ to $\ta(\pi)$.  If
$\pi$ is labelled by $f$, then the labels of these paths are
$\toppath f$ and $\botpath f$, respectively.  Equivalently, traversing
the boundary clockwise follows the top path forward and the bottom path
backward.
\end{enumerate}
Thus a diagram lies between its top and bottom paths and is tiled by
labelled cells.  If the labels of $\toppath\Delta$ and $\botpath\Delta$
are the $1$-paths $u$ and $v$, then $\Delta$ is a \emph{$(u,v)$-diagram};
if $u=v$, it is \emph{spherical} with \emph{base path} $u$.  Diagrams are
identified up to label-preserving isotopy of the plane.  The
\emph{trivial diagram} $\varepsilon(p)$ is the single path labelled $p$,
a $(p,p)$-diagram with no cells.

If $\Delta_1$ is a $(u,v)$-diagram and $\Delta_2$ is a $(v,w)$-diagram,
their \emph{concatenation} $\Delta_1\circ\Delta_2$ is obtained by
identifying $\botpath{\Delta_1}$ with $\toppath{\Delta_2}$.

Two adjacent faces form a \emph{dipole} if the upper face is labelled by
a cell $f$, the lower face is labelled by $f^{-1}$, and the bottom path
of the upper face is the top path of the lower face.  Cancelling the
dipole removes both faces and identifies the top path of the upper face
with the bottom path of the lower face.  Two diagrams are
\emph{equivalent} if they are related by finitely many insertions and
cancellations of dipoles, and a diagram is \emph{reduced} if it contains
no dipole.  By Kilibarda's fundamental lemma
\cite[Theorem~3.17]{GS97}; see also \cite[Theorem~2.6]{GS06}, every diagram is equivalent to a unique
reduced diagram.
\end{definition}

\begin{definition}[Atomic diagrams]\label{def:atomic-diagram}
An \emph{atomic diagram} is a diagram with exactly one cell.  More
explicitly, let $f\in F$ and let $a,b$ be $1$-paths such that
$\ta(a)=\io(f)$ and $\io(b)=\ta(f)$.  The atomic diagram determined by
$(a,f,b)$ has top and bottom labels
\[
a\,\toppath f\,b
\qquad\text{and}\qquad
a\,\botpath f\,b.
\]
Its inverse is the atomic diagram determined by $(a,f^{-1},b)$.
\end{definition}

Note that every non-trivial diagram is a finite concatenation of atomic
diagrams.

\begin{definition}[Homotopy]\label{def:homotopy}
Two $1$-paths $u$ and $v$ are \emph{homotopic} if there exists a
$(u,v)$-diagram over $K$.  Homotopic paths have the same endpoints.
\end{definition}

\begin{definition}[The diagram groupoid and diagram groups;
{\cite[\S2]{GS06}}]\label{def:dg}
The \emph{diagram groupoid} $\DG(K)$ consists of all reduced diagrams
over $K$.  If $\Delta_1$ is a reduced $(p,q)$-diagram and $\Delta_2$ is
a reduced $(q,r)$-diagram, their product is the reduced form of
$\Delta_1\circ\Delta_2$; otherwise their product is undefined.  The
identity associated with a $1$-path $p$ is $\varepsilon(p)$, and the
inverse of a diagram is its reflection across a horizontal line.

For a non-empty $1$-path $p$, the reduced $(p,p)$-diagrams form a group,
the \emph{diagram group over $K$ with base path $p$}, denoted
$\DG(K,p)$.
\end{definition}

Changing the base path within its homotopy class does not change the
diagram group.

\begin{lemma}[{\cite[Corollary~3.5]{GS06}}]\label{lem:homotopic-base}
If $p$ and $q$ are homotopic $1$-paths in a directed $2$-complex $K$,
then $\DG(K,p)\cong\DG(K,q)$.
\end{lemma}

\begin{proof}
If $\Gamma$ is a reduced $(p,q)$-diagram, conjugation by $\Gamma$ gives
an isomorphism
\[
\Delta\longmapsto \Gamma^{-1}\Delta\Gamma
\]
from the group with base $p$ to the group with base $q$.
\end{proof}

\begin{remark}[Semigroup presentations]\label{rem:presentation}
In diagram-group theory a semigroup presentation may be regarded as an
indexed string rewriting system
\[
\Pres=\langle X\mid u_i\to v_i\ (i\in I)\rangle .
\]
Different indices may specify the same pair of words, and a rule
$u_i\to u_i$ is allowed.  Such a system determines a one-vertex directed
$2$-complex with one edge for each letter of $X$ and one positive cell
$f_i$ with $\toppath{f_i}=u_i$ and $\botpath{f_i}=v_i$.  Conversely,
every one-vertex directed $2$-complex with its chosen positive cells
determines such an indexed rewriting system.  Under this correspondence
the diagrams and diagram groupoids coincide, and we write
$\DG(\Pres,p)$ for the corresponding diagram group.
\end{remark}

Identifying vertices may create new readable paths, but it does not
change the diagram groups based at paths that were already readable.

\begin{lemma}[Identifying vertices; cf.\ {\cite[Remark~2.3]{GS06}}]
\label{lem:vertices}
Let $\overline K$ be obtained from a directed $2$-complex $K$ by
identifying vertices, while retaining the same edges, positive cells,
and top and bottom paths.  For every non-empty $1$-path $p$ of $K$, the
natural map induces an isomorphism
\[
\DG(K,p)\cong\DG(\overline K,p).
\]
In particular, let $\Pres_K$ denote the indexed rewriting system obtained
by identifying all vertices of $K$.  Then
\[
\DG(K,p)\cong\DG(\Pres_K,p).
\]
\end{lemma}

\begin{proof}
Every diagram over $K$ is also a diagram over $\overline K$.  Conversely,
let $\Delta$ be a spherical diagram over $\overline K$ with base path
$p$.  Write $\Delta$ as a finite concatenation of atomic diagrams.
Starting with the $K$-path $p$, each
atomic factor replaces one path of a cell by the other at a displayed
occurrence.  By \cref{rem:readable}, the resulting intermediate path is
again a $K$-path.  Inductively every atomic factor, and hence $\Delta$,
is a diagram over $K$.  Dipoles and reduction are the same in the two
complexes, so the diagram groups agree.  The quotient may nevertheless
have additional readable paths and hence additional groups at those new
base paths.
\end{proof}

\subsubsection{The Guba--Sapir moves}

The following moves modify a directed $2$-complex without changing its
diagram groups.  They were established by Guba and Sapir
\cite[Theorem~4.1]{GS06} as directed-$2$-complex analogues of Tietze
transformations. The formulation of the moves here is closer to the formulation in \cite[Theorem~2.1]{BS}, where they are stated for
semigroup presentations.

\begin{convention}[Formal inverse cells]\label{conv:inverse-cells}
Whenever a positive cell is adjoined, deleted, or modified, its formal
inverse is adjoined, deleted, or modified simultaneously.
\end{convention}

\begin{theorem}[Guba--Sapir moves; {\cite[Theorem~4.1]{GS06}}]
\label{thm:GSmoves}
Let $K$ be a directed $2$-complex.
\begin{enumerate}
\item[\textup{(M1)}] \textup{(Adjoining an edge.)}  Given a non-empty
$1$-path $w$, adjoin a fresh edge $x$ with the same endpoints as $w$ and
a positive cell $x\to w$.  The resulting complex $K'$ satisfies
\[
\DG(K,q)\cong\DG(K',q)
\]
for every non-empty $1$-path $q$ of $K$.

\item[\textup{(M1)}$^{-1}$] \textup{(Deleting an edge.)}  Suppose $x$
is an edge such that there is a positive cell $f\colon x\to w$, where
$w$ avoids $x$, and $x$ occurs in no top or bottom path of any other
positive cell.  Delete $x$ and $f$.  If $K'$ is the resulting complex,
then
\[
\DG(K,q)\cong\DG(K',q)
\]
for every non-empty $1$-path $q$ not containing $x$.

\item[\textup{(M2)}] \textup{(Rewriting a cell.)}  Let $f$ and $g$ be
distinct positive cells.  If either path of $g$ contains an occurrence of
either $\toppath f$ or $\botpath f$, replace that occurrence by the other
path of $f$, leaving the other path of $g$ unchanged.  The resulting
complex $K'$ satisfies
\[
\DG(K,q)\cong\DG(K',q)
\]
for every non-empty $1$-path $q$ of $K$.
\end{enumerate}
All three moves leave the vertex set unchanged.
\end{theorem}

\begin{proof}
Moves (M1) and (M1)$^{-1}$ are \cite[Theorem~4.1(1)]{GS06} read in the
two directions.  For (M2), let $K_1$ consist of the $1$-skeleton of $K$
together with the positive cell $g$, and let $K_2$ consist of the
$1$-skeleton together with all other positive cells.  In $K_2$, the side
of $g$ being changed is homotopic to the path obtained after the
replacement, by one application of the cell $f$; the other side of $g$
is homotopic to itself.  Replacing $g$ by the modified cell therefore
preserves all diagram groups by \cite[Theorem~4.1(2)]{GS06}.
\end{proof}

A finite sequence $\Sigma$ of moves (M1), (M1)$^{-1}$ and (M2)
transforming $K$ into $K^{*}$ is a \emph{GS sequence}; two complexes
connected by such a sequence are \emph{GS-equivalent}.

\begin{definition}[Edge mapping of a GS sequence]\label{def:edge-map}
Let $\Sigma=(\sigma_1,\dots,\sigma_s)$ be a GS sequence from $K$ to
$K^{*}$.  For every edge $e$ of $K$, define a $1$-path
$\phi_\Sigma(e)$ over $K^{*}$ as follows.  Start with $\phi(e)=e$ and
process the moves in order.  Moves (M1) and (M2) leave $\phi$ unchanged.
When (M1)$^{-1}$ deletes an edge $x$ with cell $x\to w$, replace every
occurrence of $x$ in every current value of $\phi$ by $w$.  Extend
$\phi_\Sigma$ to $1$-paths by concatenation.  Each substitution preserves
endpoints, so $\phi_\Sigma(p)$ has the same endpoints as $p$.  The map
depends on the sequence $\Sigma$, not merely on $K$ and $K^{*}$.
\end{definition}

\begin{lemma}[Transport of bases]\label{lem:transport}
Let $\Sigma$ be a GS sequence from $K$ to $K^{*}$ and let $p$ be a
non-empty $1$-path of $K$.  Then
\[
\DG(K,p)\cong\DG\bigl(K^{*},\phi_\Sigma(p)\bigr).
\]
\end{lemma}

\begin{proof}
It suffices to consider one move.  For (M1) and (M2), the edge mapping is
the identity and the assertion is part of \cref{thm:GSmoves}.  Suppose
(M1)$^{-1}$ deletes $x$ with cell $x\to w$, and let $p'$ be obtained
from $p$ by replacing every occurrence of $x$ by $w$.  Replacing these
occurrences one at a time gives a concatenation of atomic diagrams, one
for each occurrence, from $p$ to $p'$.  Thus $p$ and $p'$ are homotopic, so
$\DG(K,p)\cong\DG(K,p')$ by \cref{lem:homotopic-base}.  Since $p'$ avoids
$x$, (M1)$^{-1}$ gives $\DG(K,p')\cong\DG(K',p')$.
\end{proof}

\subsubsection{Elementary invariants of the moves}

\begin{lemma}[GS invariants]\label{lem:GSinv}
Let $K'$ be obtained from a directed $2$-complex $K$ by one of the moves
\textup{(M1)}, \textup{(M1)}$^{-1}$, or \textup{(M2)}.  With the common
vertex set $V$:
\begin{enumerate}
\item a vertex has an incoming edge in $K$ if and only if it has one in
$K'$, and likewise for outgoing edges;
\item for all $u,v\in V$, there is a $1$-path from $u$ to $v$ in $K$ if
and only if there is one in $K'$.
\end{enumerate}
Hence the moves preserve the strongly connected components, the vertices
with no incoming edge, and the vertices with no outgoing edge.
\end{lemma}

\begin{proof}
An (M2) move leaves the $1$-skeleton unchanged, and both assertions are
symmetric in $K$ and $K'$, so it suffices to treat (M1).  Suppose it
adjoins $x$ with a positive cell $x\to w$.  Only $\io(x)$ and $\ta(x)$ gain an
incident edge, but $\io(w)$ already has an outgoing edge and $\ta(w)$ an
incoming edge.  Every path of $K$ remains a path of $K'$, while replacing
each occurrence of $x$ in a path of $K'$ by $w$ gives a path of $K$ with
the same endpoints.
\end{proof}

\subsubsection{Tree-like complexes; active and dummy edges}

\begin{definition}[Tree-like complexes]\label{def:treelike}
A directed $2$-complex is \emph{tree-like} if the top path of every
positive cell is a single edge, its bottom path has length at least two,
and distinct positive cells have distinct top edges.  An edge is
\emph{active} if it is the single-edge top path of a positive cell;
every other edge is \emph{dummy}.  Active and dummy edges may both occur
in bottom paths.
\end{definition}

This is the directed-$2$-complex analogue of the tree-like semigroup
presentations of Farley and Hughes
\cite[Definition~4.1]{FarleyHughes}; presentations of this form were used
earlier in \cite[\S5]{GS99}.

In a tree-like complex no dummy edge is a whole path of a positive cell,
so it cannot be deleted immediately by (M1)$^{-1}$.  A GS sequence may,
however, exchange a dummy edge $d$ for a fresh parallel edge: adjoin
$x\to d$, rewrite the other occurrences of $d$ to $x$, and delete $d$.
The GS moves do not in general preserve tree-likeness.

Every GS sequence used in the proof of \cref{thm:reduction} begins and
ends at a tree-like complex, although its intermediate complexes need not
be tree-like.

\subsection{Thompson's groups \texorpdfstring{$F_n$}{Fn}}\label{subsec:Fn}

For $n\geq2$, $F_n$ denotes the $n$-ary Thompson group: the group of
piecewise-linear orientation-preserving homeomorphisms of $[0,1]$ with
finitely many breakpoints, all in $\Z[1/n]$, and slopes integral powers
of $n$ \cite{Brown,CFP}.  Thus $F_2=F$.  The groups $F_n$ are of type
$F_\infty$ and have free-abelian abelianization of rank $n$
\cite{Brown}; in particular they are pairwise non-isomorphic.  Guba and
Sapir identified them as diagram groups.

\begin{theorem}[{\cite{GS97}}]\label{thm:Fn-diagram}
For every $n\geq2$,
\[
F_n\cong\DG\bigl(\langle x\mid x=x^n\rangle,x\bigr).
\]
\end{theorem}

The relation with the usual tree-pair model can be read directly from a
reduced spherical diagram over $\langle x\mid x=x^n\rangle$.  Its longest
directed path from source to sink passes through all vertices and
separates the diagram into an upper part whose cells are positive and a
lower part whose cells are negative \cite[\S5]{GS99}.  Replacing each
cell in the upper part by an $n$-caret gives a rooted $n$-ary tree;
reflecting the lower part gives a second such tree.
The edges of the separating path correspond to the common ordered set of
leaves. Thus reduced
spherical diagrams give the standard reduced pairs of $n$-ary trees; see
\cite{Brown,CFP}.

\section{Peelability of irreducible diagrams via swap moves}\label{sec:swap}

In this section we show that every irreducible fast set of positive bumps
can be replaced, without changing the group it generates, by a fast set
whose dynamical diagram is \emph{peelable}: it remains irreducible under
repeated deletion of the rightmost bump.  On the combinatorial level this
amounts to transforming an irreducible dynamical diagram by swap moves.
All the dynamics is confined to the Gap Swap Lemma
(\cref{lem:gap-swap}); the rest is combinatorics.

Throughout this section, the bumps of a dynamical diagram $D$ on $n$
bumps are numbered so that
\[
R_1<R_2<\dots<R_n ,
\]
the \emph{$R$-order}.  The last foot of $D$ is a right foot, namely
$R_n$ (a left foot is never last, as it precedes the right foot of the
same bump); we call $b_n$ the \emph{last bump}.  For $1\leq k\leq n$ the
\emph{$k$-th prefix} $P_k$ is the sub-diagram on the bumps
$b_1,\dots,b_k$, and $\widehat D:=P_{n-1}$ is the \emph{deleted diagram}.

\subsection{Peelability}

\begin{definition}[Peelable]\label{def:peelable}
A dynamical diagram $D$ is \emph{peelable} if every prefix $P_k$ with
$k\geq2$ is irreducible.  A diagram on at most one bump is peelable, and
for $n\geq2$ a peelable diagram is irreducible and its deleted diagram
$\widehat D$ is again peelable.  A fast set is called peelable if its
dynamical diagram is.
\end{definition}

\begin{lemma}[Recognizing peelability]\label{lem:recognize}
A dynamical diagram $D$ is peelable if and only if, in the $R$-order,
every bump $b_k$ with $k\geq2$ crosses some earlier bump $b_j$, $j<k$;
equivalently, for every $k\geq 2$ there is $j<k$ with $L_j<L_k<R_j$.
\end{lemma}

\begin{proof}
If all prefixes are irreducible, then for $k\geq2$ the vertex $b_k$ is
not isolated in the connected graph $\GB(P_k)$, so $b_k$ crosses some
$b_j$ with $j<k$.  Conversely, if every $b_k$ ($k\geq2$) crosses an
earlier bump, then by induction each $\GB(P_k)$ is connected: it is
obtained from the connected graph $\GB(P_{k-1})$ by attaching the vertex
$b_k$ along at least one edge.

Finally, for $j<k$ we have $R_j<R_k$, so $b_j$ and $b_k$ cross if and
only if $L_j<L_k<R_j$.
\end{proof}

\subsection{The swap move}\label{subsec:swap-move}

We now make systematic the conjugation phenomenon mentioned in the
introduction.  The local operation introduced below, called a
\emph{swap}, changes the dynamical diagram of a fast generating set while
preserving both fastness and the subgroup generated by the bumps.

Let $D$ be a dynamical diagram and $b$ one of its bumps.  In any
fast realization, a foot of another bump that meets $\supp(b)$ is
contained in the gap of $b$: it is disjoint from the two feet of $b$, and
the gap is the remaining part of the support.  Hence the feet strictly
between $L_b$ and $R_b$ are exactly the feet contained in the gap of $b$.
Listed in order they form the \emph{gap word}
$G_b=f_1f_2\cdots f_r$ of $b$, a purely combinatorial feature of $D$.
We say a bump \emph{owns} its two feet.

\begin{definition}[Closed cut; swap]\label{def:cut}
A \emph{cut} of the gap word $G_b=XY$ is a factorization into an initial
block $X=f_1\cdots f_t$ and a terminal block $Y=f_{t+1}\cdots f_r$
\textup{(}$0\le t\le r$\textup{)}.  The cut is \emph{closed} if no bump
owns a foot in $X$ and a foot in $Y$.  The \emph{swap} at a closed
cut replaces the diagram $D$ \textup{(}with gap word $XY$ at
$b$\textup{)} by the diagram $D'$ which is identical except that the gap
word of $b$ is $YX$: the two blocks are interchanged, the internal order
of each block and all other feet being unchanged.  Two diagrams are
\emph{swap-equivalent} if a finite sequence of swaps carries one to the
other.
\end{definition}

\begin{lemma}[Gap Swap Lemma]\label{lem:gap-swap}
Let $B$ be a fast realization of the dynamical diagram $D$, let
$b\in B$ have support $(a,c)$, and let $G_b=XY$ be a closed cut with
$Y$ non-empty \textup{(}if $Y=\emptyset$ the swap is trivial\textup{)}.
Let $M\subseteq B$ be the set of owners of the feet of $Y$, and let $q$
be the infimum of the leftmost foot of $Y$.  Then the set
\[
B' \;=\; \bigl(B\setminus M\bigr)\ \cup\ \{\,b^{-1}yb \mid y\in M\,\},
\]
with the marking obtained by transporting the markers of the members of
$M$ by $b^{-1}$, moving the marker of $b$ to $m'=b^{-1}(q)$, and keeping
all other markers, is a fast realization of the swapped diagram $D'$, and
$\langle B'\rangle=\langle B\rangle$.
\end{lemma}

\begin{proof}
Let $m$ be the old marker of $b$.  The
conjugators lie in $\langle B\rangle$, so
$\langle B'\rangle\subseteq\langle B\rangle$; and $b\in B'$ (the feet
of $b$ do not lie in its own gap word, so $b\notin M$), whence
$y=b\,(b^{-1}yb)\,b^{-1}\in\langle B'\rangle$ for $y\in M$ and
$\langle B'\rangle=\langle B\rangle$.

\emph{Feet of the conjugated bumps.}
For $y\in M$ with marker $m_y$, the conjugate $y'=b^{-1}yb$ is again a
positive bump, with support $b^{-1}(\supp(y))$; the transported marker
$b^{-1}(m_y)$ lies in this support, so the marking of $B'$ is a valid
marking, and the feet of $y'$ are
\begin{align*}
L_{y'}&=\bigl(b^{-1}(a_y),\,b^{-1}(m_y)\bigr)=b^{-1}(L_y),\\
R_{y'}&=\bigl[\,y'(b^{-1}(m_y)),\,b^{-1}(c_y)\bigr)
      =\bigl[\,b^{-1}(y(m_y)),\,b^{-1}(c_y)\bigr)=b^{-1}(R_y),
\end{align*}
because $y'\circ b^{-1}=b^{-1}\circ y$ and $b^{-1}$ is an
orientation-preserving homeomorphism.  So every foot of every conjugated
bump is the $b^{-1}$-image of the corresponding old foot.  By closedness
of the cut, no $y\in M$ owns a foot in $X$; hence each foot of each $y\in M$ is either a foot of $Y$ or
disjoint from $(a,c)$, and in the latter case it is fixed
pointwise by $b^{-1}$.  Thus the feet that actually move are exactly the
feet of $Y$, and they move by $b^{-1}$.

\emph{Positions.}
All feet of $Y$ lie in $[\,q,\ b(m))$: they lie in the gap
$[\,m,\,b(m))$ of $b$ by the definition of the gap word, and at or to the right of
$q$ by the definition of $q$ and the ordering of the gap feet.  Their
images under $b^{-1}$ are pairwise disjoint intervals contained in
$[\,m',\,m)$, in the same relative order, where $m'=b^{-1}(q)$.  The feet
of $X$ lie in $[\,m,\,q)$: they lie in the gap and strictly to the left
of the leftmost foot of $Y$.  The new feet of $b$ are $L'_b=(a,\,m')$ and
$R'_b=[\,b(m'),\,c)=[\,q,\,c)$.

\emph{Fastness.}
By the previous paragraph, the pairwise disjoint intervals $(a,m')$,
$[m',m)$, $[m,q)$ and $[q,c)$ contain, respectively, $L'_b$, the moved
feet (pairwise disjoint, being $b^{-1}$-images of pairwise disjoint
feet), the feet of $X$, and $R'_b$; every other foot is unmoved and
disjoint from $(a,c)$, and unmoved feet were
pairwise disjoint already.  So the new feet are pairwise disjoint and
$B'$ is fast.

\emph{The diagram.}
Within $\supp(b)$ the feet now read: $L'_b$, then the feet of $Y$ (in
$[m',m)$, in their old relative order), then the feet of $X$
(unchanged, in $[m,q)$), then $R'_b$.  Every foot outside $(a,c)$ is
unchanged, and moved feet remain inside $(a,c)$, so all order relations
with the outside are unchanged.  Hence the dynamical diagram of
$(B',\text{new marking})$ is exactly $D'$.
\end{proof}

\begin{corollary}\label{cor:swap-groups}
If $D$ and $D'$ are swap-equivalent dynamical diagrams, then for every
fast realization $B$ of $D$ there is a fast realization $B'$ of $D'$
with $\langle B'\rangle=\langle B\rangle$.  In particular, the groups
generated by fast realizations of $D$ and of $D'$ are isomorphic.
\end{corollary}

\begin{proof}
Apply \cref{lem:gap-swap} step by step along a sequence of swaps
carrying $D$ to $D'$: each step replaces some members of the current
set by their conjugates inside the group it generates, and changes
markers, so it preserves the generated group.  The last statement
follows from \cref{thm:diagram-isomorphism}.
\end{proof}

The geometric effect of a swap is illustrated in \cref{fig:swap}.

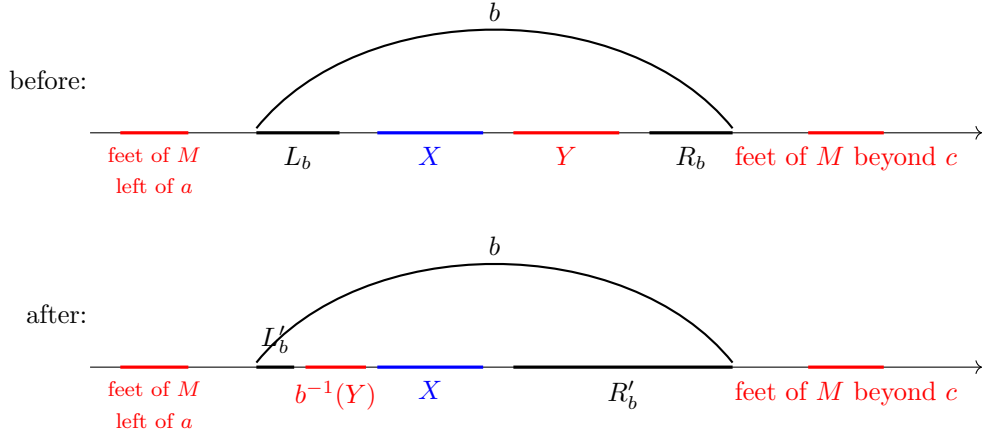
\begin{figure}[ht]
\centering
\begin{tikzpicture}[scale=1.0]
  \begin{scope}
    \draw[->] (-2.1,0)--(9.7,0);
    \draw[very thick,red] (-1.7,0)--(-0.8,0)
      node[midway,below=1pt,align=center] {\scriptsize feet of $M$\\[-1pt]\scriptsize left of $a$};
    \draw[very thick] (0.1,0)--(1.2,0) node[midway,below=1pt] {\small $L_b$};
    \draw[very thick,blue] (1.7,0)--(3.1,0) node[midway,below=1pt] {\small $X$};
    \draw[very thick,red] (3.5,0)--(4.9,0) node[midway,below=1pt] {\small $Y$};
    \draw[very thick] (5.3,0)--(6.4,0) node[midway,below=1pt] {\small $R_b$};
    \draw[very thick,red] (7.4,0)--(8.4,0) node[midway,below=1pt] {\small feet of $M$ beyond $c$};
    \draw[thick] (0.1,0.06) .. controls (1.5,1.8) and (5.0,1.8) .. (6.4,0.06);
    \node at (3.25,1.6) {\small $b$};
    \node[left] at (-2.0,0.7) {\small before:};
  \end{scope}
  \begin{scope}[yshift=-3.1cm]
    \draw[->] (-2.1,0)--(9.7,0);
    \draw[very thick,red] (-1.7,0)--(-0.8,0)
      node[midway,below=1pt,align=center] {\scriptsize feet of $M$\\[-1pt]\scriptsize left of $a$};
    \draw[very thick] (0.1,0)--(0.6,0) node[midway,above=1pt] {\small $L'_b$};
    \draw[very thick,red] (0.75,0)--(1.55,0) node[midway,below=1pt] {\small $b^{-1}(Y)$};
    \draw[very thick,blue] (1.7,0)--(3.1,0) node[midway,below=1pt] {\small $X$};
    \draw[very thick] (3.5,0)--(6.4,0) node[midway,below=1pt] {\small $R'_b$};
    \draw[very thick,red] (7.4,0)--(8.4,0) node[midway,below=1pt] {\small feet of $M$ beyond $c$};
    \draw[thick] (0.1,0.06) .. controls (1.5,1.8) and (5.0,1.8) .. (6.4,0.06);
    \node at (3.25,1.6) {\small $b$};
    \node[left] at (-2.0,0.7) {\small after:};
  \end{scope}
\end{tikzpicture}
\caption{The swap $XY\to YX$ at the bump $b$.  The bump $b$ itself, and
hence its support $(a,c)$ (spanned by the arc, which joins the left
endpoint of $L_b$ to the right endpoint of $R_b$), is unchanged.  The
owners of the feet of $Y$ are conjugated by $b^{-1}$; their feet inside
the gap of $b$ (red) are dragged into the old left-foot zone of $b$,
while their feet beyond $\supp(b)$ do not move and appear identically
before and after.  The right foot of $b$ expands leftwards over the
vacated zone, and the marker of $b$ moves left.}
\label{fig:swap}
\end{figure}

\subsection{Merging the components}\label{subsec:merge}

Fix an irreducible diagram $D$ with $n\geq2$ bumps in $R$-order, let
$b_n$ be the last bump, and recall that the \emph{deleted diagram}
$\widehat D=P_{n-1}$ is the sub-diagram of $D$ on the bumps
$b_1,\dots,b_{n-1}$, obtained by deleting $b_n$.  Since $R_n$ is the
last foot, no bump crosses $b_n$ from the right; thus the bumps crossing
$b_n$ --- the \emph{crossers} --- are exactly those $x$ with
\[
L_x<L_n<R_x\ (<R_n),
\]
and the feet strictly between $L_n$ and $R_n$ (the gap word of $b_n$)
are the right feet of the crossers together with both feet of the bumps
nested in $b_n$.

Throughout this subsection, ``component'' means a connected component of
$\GB(\widehat D)$, the crossing graph of the deleted diagram.  For a set
$S$ of bumps with connected crossing graph, the \emph{reach} of $S$ is
the position interval $I_S=(\min_S L,\ \max_S R)$.  Crossing bumps have
overlapping spans, so the union of the spans of the members of $S$ is an
interval, namely all of $I_S$: \emph{the reach is covered by the spans of
the members of $S$}.

The following lemma is the engine of all the structural arguments.

\begin{lemma}\label{lem:spread}
Let $D$ be a dynamical diagram, let $S$ be a set of bumps of $D$ whose
crossing graph is connected, and let $w\notin S$ be a bump crossing no
member of $S$.  If some member of $S$ is nested in $w$, then every member
of $S$ is nested in $w$; in that case $I_S\subseteq\spn(w)$, and in
particular $L_w<\min_S L$ and $R_w>\max_S R$.
\end{lemma}

\begin{proof}
Let $v\in S$ be nested in $w$ and let $z\in S$ be arbitrary; choose a
path of crossing bumps $v=w_0,w_1,\dots,w_k=z$ inside $S$.  Inductively
assume $w_i$ is nested in $w$, i.e.,
$\spn(w_i)\subset\spn(w)$.  Since $w_{i+1}$ crosses $w_i$, the spans of
$w_{i+1}$ and $w$ overlap.  They do not cross ($w$ crosses no member of
$S$).  If $\spn(w)\subset\spn(w_{i+1})$, then
$\spn(w_i)\subset\spn(w)\subset\spn(w_{i+1})$, so $w_{i+1}$ nests $w_i$,
contradicting that they cross.  Hence
$\spn(w_{i+1})\subset\spn(w)$, completing the induction.  All feet of
members of $S$ therefore lie strictly inside $\spn(w)$, which gives
$I_S\subseteq\spn(w)$ and the strict inequalities (the feet $L_w$,
$\min_S L$, $\max_S R$, $R_w$ being pairwise distinct).
\end{proof}

\begin{lemma}[Structure of the components]\label{lem:structure}
Suppose $D$ is irreducible with $n\geq3$ bumps and $\GB(\widehat D)$ is
disconnected.  Then:
\begin{enumerate}
\item every component contains a crosser of $b_n$;
\item if $x$ is a crosser in a component $K$ and $y$ is a crosser in a
different component, then one of $\spn(x),\spn(y)$ strictly contains the
other; if $\spn(x)\subset\spn(y)$, then every member of $K$ is nested in
$y$, and $I_K\subseteq\spn(y)$;
\item for distinct components, define $K\prec K'$ if some crosser
$y\in K'$ satisfies $I_K\subseteq\spn(y)$.  Then $\prec$ is a strict
linear order, and in this definition ``some'' may be replaced by
``every''.  In particular there is a unique minimum component, called
the \emph{innermost} component.
\end{enumerate}
\end{lemma}

\begin{proof}
For (1), $\GB(D)$ is connected, so every component of
$\GB(\widehat D)$ contains a neighbour of $b_n$.

For (2), the spans of two crossers both contain $L_n$.  Since crossers in
different components do not cross, their spans are strictly nested.  If
$\spn(x)\subset\spn(y)$, then $x\in K$ is nested in $y$, while $y$
crosses no member of $K$; \cref{lem:spread} gives the conclusion.

For (3), part (2) shows that two distinct components are comparable.
They cannot be comparable in both directions: otherwise crossers
$y\in K$ and $y'\in K'$ would give
\[
\spn(y')\subseteq I_{K'}\subseteq\spn(y)
 \subseteq I_K\subseteq\spn(y'),
\]
which is impossible.  Transitivity follows from
\[
I_K\subseteq\spn(y')\subseteq I_{K'}\subseteq\spn(y'').
\]
Thus $\prec$ is a strict linear order.  Finally suppose $K\prec K'$ and
let $y$ be any crosser in $K'$.  If $I_K\not\subseteq\spn(y)$, then
part (2), applied to $y$ and any crosser $x\in K$, gives
$\spn(y)\subset\spn(x)$ and hence $I_{K'}\subseteq\spn(x)$; this says
$K'\prec K$, a contradiction.  Therefore every crosser in $K'$ contains
$I_K$ in its span.
\end{proof}

\begin{corollary}[Innermost Dichotomy]\label{cor:innermost}
Let $D$, $\widehat D$ be as in \cref{lem:structure} and let $K$ be the
innermost component.  Then for every bump $w\notin K\cup\{b_n\}$, exactly
one of the following holds:
\begin{enumerate}
\item every member of $K$ is nested in $w$; then $I_K\subseteq\spn(w)$,
$L_w<\min_K L$ and $R_w>\max_K R$;
\item $\spn(w)\cap I_K=\emptyset$.
\end{enumerate}
\end{corollary}

\begin{proof}
Suppose $\spn(w)\cap I_K\neq\emptyset$.  Since the reach $I_K$ is covered
by the spans of the members of $K$, the span of $w$ overlaps
$\spn(v)$ for some $v\in K$.  The bump $w$ lies in a component
$K_w\neq K$ of $\GB(\widehat D)$, so $w$ crosses no member of $K$, and
$\spn(w)$, $\spn(v)$ are strictly nested.

If $\spn(w)\subset\spn(v)$, then $w$ is a member of $K_w$ nested in $v$,
and $v\notin K_w$ crosses no member of $K_w$; by \cref{lem:spread},
every member of $K_w$ is nested in $v$.  By \cref{lem:structure}(3) some
crosser $y_p\in K_w$ has $\spn(y_p)\supseteq I_K$; but then
\[
I_K\;\subseteq\;\spn(y_p)\;\subset\;\spn(v)\;\subseteq\;I_K ,
\]
a contradiction.  Hence $\spn(v)\subset\spn(w)$, and \cref{lem:spread}
(applied to $K$ and $w$) yields case (1).  The two cases are clearly
mutually exclusive.
\end{proof}

\begin{lemma}[Initial Block Lemma]\label{lem:initial-block}
Let $D$ and $K$ be as in \cref{cor:innermost}, and let $s_F$ be the last
foot of a member of $K$ occurring in the gap word of $b_n$.  Then the
feet owned by members of $K$ that occur in the gap word of $b_n$ form an
initial block of that gap word.
\end{lemma}

\begin{proof}
First, every foot $e$ with $L_n<e<s_F$ lies in $I_K$: choosing a
crosser $x\in K$ gives $\min_K L\leq L_x<L_n<e$, and $s_F$ is a foot of
a member of $K$, so $e<s_F\leq\max_K R$; hence
$e\in(\min_K L,\,\max_K R)=I_K$.

Now suppose some such foot $e$ is owned by a bump
$w\notin K\cup\{b_n\}$.  The foot $e$ is an endpoint of $\spn(w)$
lying in the open interval $I_K$, so $\spn(w)\cap I_K\neq\emptyset$.  By
\cref{cor:innermost}, $I_K\subseteq\spn(w)$ with $L_w<\min_K L$ and
$R_w>\max_K R$.  But $e\in\{L_w,R_w\}$, and both $L_w$ and $R_w$ lie
outside the interval $(\min_K L,\,\max_K R)$ containing $e$ --- a
contradiction.
\end{proof}

\begin{lemma}[Merge Lemma]\label{lem:merge}
Let $D$ be irreducible with $n\geq3$ bumps and $\GB(\widehat D)$
disconnected, let $K$ be the innermost component of $\GB(\widehat D)$,
and let $X$ be the initial block of the gap word of $b_n$ formed by the
feet owned by members of $K$ that occur in that gap word
\textup{(}\cref{lem:initial-block}\textup{)}, and let $Y$ be the rest.
Then the cut $(X)(Y)$ is closed, and the swap $XY\to YX$ at $b_n$
produces an irreducible diagram $D^{\ast}$ in which
$\GB(\widehat{D^{\ast}})$ is connected.  Moreover the swap destroys no
crossing.
\end{lemma}

\begin{proof}
\emph{Closedness.}  The owners of the feet of $X$ are members of $K$,
and \emph{all} feet of members of $K$ lying in the gap word belong to
$X$ by its definition; the owners of feet of $Y$ are not in $K$.
Hence no bump owns a foot in each block.

\emph{No crossing is destroyed.}  The swap only interchanges the blocks
$X$ and $Y$: the relative order of any two feet not separated by the cut
is unchanged, and so is the order of every foot relative to $L_n$ and
$R_n$ (only feet strictly between them are permuted, and they stay
between them); in particular the set of crossers of $b_n$ is unchanged.
A crossing could therefore only be destroyed between a member of $K$
(owning a foot of $X$) and a bump owning a foot of $Y$ --- a bump that
does not belong to $K$; bumps in distinct components of
$\GB(\widehat D)$ do not cross, so no such crossing existed.

\emph{Connectivity.}  Let $K_p\neq K$ be a component, let $y_p\in K_p$
be a crosser with $I_K\subseteq\spn(y_p)$
\textup{(}\cref{lem:structure}(3)\textup{)}, and let $x\in K$ be a
crosser.  Before the swap, $L_{y_p}<L_x<R_x<R_{y_p}$, the foot $R_x$
lies in $X$ (it lies in the gap word since $x$ is a crosser) and the
foot $R_{y_p}$ lies in $Y$ (it lies in the gap word and its owner does
not belong to $K$).  After the swap the block $X$ sits to the right of
$Y$, so $R_{y_p}$ comes before $R_x$, while the left feet are untouched:
\[
L_{y_p}\;<\;L_x\;<\;L_n\;<\;R_{y_p}\;<\;R_x ,
\]
i.e., $x$ crosses $y_p$ in $D^{\ast}$.  Thus in
$\GB(\widehat{D^{\ast}})$ the component of $K$ acquires an edge to every
other component, while no edge is destroyed; hence
$\GB(\widehat{D^{\ast}})$ is connected.  Finally $\GB(D^{\ast})$ is
connected because the crossers of $b_n$ are unchanged, and there is at
least one ($b_n$ is not isolated in the connected graph $\GB(D)$).
\end{proof}

\subsection{Lifting swaps and the peelability theorem}

\begin{lemma}[Free Lifting Lemma]\label{lem:free-lift}
Let $D$ be a dynamical diagram with last bump $b_n$ and deleted diagram
$\widehat D$.  Every swap on $\widehat D$ \textup{(}at a bump $w\neq b_n$
and a closed cut of its gap word in $\widehat D$\textup{)} extends to a
swap on $D$ at $w$ whose effect on the feet of $\widehat D$ is the
given one.
\end{lemma}

\begin{proof}
The gap word of $w$ in $D$ is its gap word in $\widehat D$ with possibly
the foot $L_n$ inserted ($R_n$, the last foot of $D$, lies in no gap
word).  Cut the gap word of $w$ in $D$ at the position of the given cut,
putting $L_n$ --- if it lies in the gap word --- into whichever block
contains it (if it sits exactly at the cut, into either block).  The
owner $b_n$ of $L_n$ has no other foot in the gap word of $w$, so the
extended cut is closed, and interchanging its blocks restricts to the
given interchange on the feet of $\widehat D$.
\end{proof}

\begin{lemma}[The last bump stays attached]\label{lem:attached}
Let $D$ be a dynamical diagram whose last bump $b_n$ crosses at least
one bump, and let $D'$ be obtained from $D$ by a swap at a bump
$w\neq b_n$.  Then the last bump of $D'$ is still $b_n$, and $b_n$
crosses at least one bump of $D'$: either the set of crossers of $b_n$
is unchanged, or $w$ itself crosses $b_n$ in both $D$ and $D'$.
\end{lemma}

\begin{proof}
A swap at $w$ permutes feet inside the position interval
$(L_w,R_w)$, which lies to the left of $R_n$; so $R_n$ remains the last
foot.  If $L_n\notin(L_w,R_w)$, then no foot changes sides with
respect to $L_n$ or $R_n$, so the set of crossers of $b_n$ --- the bumps
with right foot in $(L_n,R_n)$ and left foot before $L_n$ --- is
unchanged, and it is non-empty by hypothesis.  If $L_n\in(L_w,R_w)$,
then $L_w<L_n<R_w<R_n$, so $w$ crosses $b_n$; and since the feet
$L_w,R_w$ of the bump at which the swap is performed do not move, $w$
still crosses $b_n$ after the swap.
\end{proof}

\begin{theorem}[Peelability]\label{thm:peelable}
Every irreducible dynamical diagram is swap-equivalent to a peelable
diagram.  Consequently, by \cref{cor:swap-groups}, for every
irreducible fast set $B$ of $n$ positive bumps there is a fast set $B'$
of $n$ positive bumps realizing a peelable diagram with
$\langle B'\rangle=\langle B\rangle$.
\end{theorem}

\begin{proof}
Induction on $n$.  For $n\leq2$ every irreducible diagram is already
peelable.  Let $n\geq3$ and let $D$ be irreducible with last bump
$b_n$.

\emph{Step 1 (merge).}
If $\GB(\widehat{D})$ is disconnected, apply the single swap of
\cref{lem:merge}; rename the result $D$.  In either case $\widehat{D}$
is now irreducible (with $n-1\geq2$ bumps) and $b_n$ crosses at least one
bump: after a merge this is part of \cref{lem:merge}, and if no merge
was needed it holds because $b_n$ is not an isolated vertex of the
connected graph $\GB(D)$.

\emph{Step 2 (recurse and lift).}
By the induction hypothesis there is a finite sequence of swaps
transforming $\widehat{D}$ into a peelable diagram $\widehat{D}'$.
Lift the sequence, swap by swap, to $D$ by \cref{lem:free-lift}; let
$D'$ be the resulting diagram.  By construction, deleting $b_n$ from
$D'$ gives $\widehat{D}'$.

\emph{Step 3 (conclusion).}
Each lifted swap is performed at a bump other than $b_n$, so by
\cref{lem:attached}, applied once per swap, $b_n$ remains last and
crosses at least one other bump in $D'$.  Since $b_n$ remains last, the
first $n-1$ bumps of $D'$ in the current $R$-order are precisely the
bumps of $\widehat{D}'$ in their $R$-order.  Hence the first $n-1$
prefixes of $D'$ are the prefixes of $\widehat{D}'$, which are
irreducible since $\widehat{D}'$ is peelable; and the $n$-th prefix,
$D'$ itself, is irreducible because $\GB(\widehat{D}')$ is connected
and $b_n$ is joined to it.  By \cref{def:peelable}, $D'$ is peelable.
\end{proof}

\section{Cores and dilated cores of \texorpdfstring{$F_n$}{Fn}}\label{sec:cores}

In this section we first recall the core of a diagram group, a
construction due to Guba and Sapir which first appeared in the work of
the author and M.~Sapir \cite{GolanSapirTAMS}, and present the core of
$F_n$ as a directed $2$-complex;
its diagram group is $F_n$.  We then introduce \emph{dilated cores}, the
complexes at which the reduction of \cref{sec:reduction} will arrive, and
show that their diagram groups are all isomorphic to $F_n$.

\subsection{The core of a diagram group}\label{subsec:core}

Let $K$ be a directed $2$-complex, let $p$ be a non-empty base path, and
let $H\leq\DG(K,p)$ be generated by reduced spherical diagrams
$\Delta_1,\dots,\Delta_s$.  The \emph{(Stallings) $2$-core}
$\mathcal L(H)$ is obtained by identifying the top and bottom paths of
each $\Delta_i$, wedging the resulting complexes along the images of
their base paths, and repeatedly folding equally labelled cells that
share their top or bottom paths \cite[\S3]{GolanSapirTAMS}.  The common
image of the wedged base paths is the base path $p_H$ of the core.  The
result is independent of the order of the foldings and of the chosen
finite generating set \cite[Lemma~3.20 and Remarks~3.21, 3.23]
{GolanSapirTAMS}.

The \emph{closure} $\Cl(H)$ consists of the elements of $\DG(K,p)$ whose
reduced diagrams are accepted by $\mathcal L(H)$; the subgroup $H$ is
\emph{closed} if $H=\Cl(H)$.  The immersion of the core into $K$
induces an isomorphism
\[
\Cl(H)\cong\DG\bigl(\mathcal L(H),p_H\bigr).
\]
Consequently every closed subgroup is isomorphic to the diagram group of
its core.  We use this in the form established for $F_n$ in
\cite{GGSFn}.

\subsection{The core of \texorpdfstring{$F_n$}{Fn}}\label{subsec:coreFn}

By \cref{thm:Fn-diagram}, $F_n$ is the diagram group of the one-vertex,
one-edge complex $\langle x\mid x=x^n\rangle$ with base path $x$.  We
first define the directed $2$-complex that will be its core.

\begin{definition}[The complex $K_n^{+}$]\label{def:coreFn}
Fix $n\geq2$ and read indices modulo $n-1$.  The directed $2$-complex
$K_n^{+}$ has the following positive cells and $1$-skeleton.
\begin{itemize}
\item \textbf{Vertices:} $\io$, $\ta$, and cycle vertices
$0,1,\dots,n-2$.
\item \textbf{Edges:} cycle edges $e_t\colon t\to t+1$ for
$t\in\Z/(n-1)$; a left edge $\eL\colon\io\to1$; a right edge
$\eR\colon0\to\ta$; and a base edge $r\colon\io\to\ta$.
\item \textbf{Positive cells:}
\[
r=\eL e_1e_2\cdots e_{n-2}\eR,
\qquad
\eL=\eL e_1e_2\cdots e_{n-2}e_0,
\qquad
\eR=e_0e_1\cdots e_{n-2}\eR,
\]
\[
e_t=e_t\,e_{t+1}e_{t+2}\cdots e_{n-2}e_0e_1\cdots e_t
\qquad(t\in\Z/(n-1)).
\]
\end{itemize}
For $n=2$ there is one cycle vertex $0$, both boundary edges attach to
$0$, there is one cycle edge $e_0$ (a loop), and the cells are
\[
r=\eL\eR,\qquad \eL=\eL e_0,\qquad
\eR=e_0\eR,\qquad e_0=e_0e_0.
\]
\end{definition}

\begin{theorem}\label{thm:core-Fn}
For every $n\geq2$,
\[
F_n\cong\DG\bigl(K_n^{+},r\bigr).
\]
\end{theorem}

\begin{proof}
By \cite[Remark~5.2]{GGSFn}, the complex $K_n^{+}$ is the core of $F_n$
as a subgroup of itself.  The whole diagram group is closed, so the
claim follows from the core construction.
\end{proof}

The base edge $r$ occurs only in its own positive cell, and the other
path of that cell avoids $r$.

\begin{corollary}\label{cor:standard-core}
Let $K_n$ be obtained from $K_n^{+}$ by deleting $r$ and its positive
cell.  Then
\[
F_n\cong\DG\bigl(K_n,\eL e_1\cdots e_{n-2}\eR\bigr).
\]
\end{corollary}

\begin{proof}
Delete $r$ by (M1)$^{-1}$ and apply \cref{lem:transport} to the base
path.
\end{proof}

We call $K_n$ the \emph{standard core} of $F_n$.  It is tree-like with
all edges active, $\io$ is its unique source, $\ta$ its unique sink, and
its inner $1$-skeleton is a directed simple cycle through all cycle
vertices.

\subsection{Dilated cores}\label{subsec:dilated}

\begin{definition}[Dilated core]\label{def:dilated-core}
Fix $n\geq2$ and $k\geq0$.  A \emph{dilated core of $F_n$ with $k$ dummy
edges} is a directed $2$-complex $K$ of the following shape.
\begin{itemize}
\item \textbf{Vertices:} an initial vertex $\io$, a terminal vertex
$\ta$, and a set $\Vin$ of $n+k-1$ \emph{inner} vertices.
\item \textbf{Edges:} a set of $n-1+k$ \emph{inner} edges forming a
single directed simple cycle through all of $\Vin$ (for
$|\Vin|=1$, a single loop), partitioned into $n-1$ \emph{active} edges
and $k$ \emph{dummy} edges; a \emph{left edge} $\eL\colon\io\to u$ and a
\emph{right edge} $\eR\colon v\to\ta$ for some $u,v\in\Vin$ (the
\emph{attachments}).
\item \textbf{$2$-cells:} writing $C(w)$ for the positive path
traversing the full inner cycle once, starting and ending at
$w\in\Vin$, the cells are
\[
e=e\cdot C(\ta(e))\ \ \text{for every active inner edge }e,\qquad
\eL=\eL\cdot C(u),\qquad
\eR=C(v)\cdot\eR .
\]
Dummy edges are the top edge of no cell, so $K$ is tree-like and the
terminology agrees with \cref{def:treelike}.
\end{itemize}
A dilated core with $k=0$ is called a \emph{shifted core}.  The standard
core $K_n$ is the shifted core with attachments $u=1$, $v=0$ (indices
modulo $n-1$; for $n=2$ the cycle is a single loop at the vertex $0$ and
$u=v=0$).
\end{definition}

The standard core and a typical dilated core are illustrated in
\cref{fig:cores}.

\begin{remark}[Cycle convention]\label{rem:cycle-convention}
On a simple cycle the two natural forms of the inner relation coincide
letter by letter: if $e$ is an inner edge then
\[
e\cdot C(\ta(e))\;=\;C(\io(e))\cdot e
\]
as positive words, both spelling ``$e$, then once around the cycle''.
We use this identity silently.
\end{remark}

\begin{remark}[Determinism]\label{rem:determinism}
In a directed simple cycle every vertex has exactly one outgoing cycle
edge, so a positive path in the cycle is completely determined by its
initial vertex and its length.  Consequently, whenever a word $W$ in the
edges of the cycle is known to spell a path from $a$ to $b$ of length
$\ell$, then $W$ is \emph{letter-for-letter equal} to the unique such
path; if $\ell=\ell_0+mL$, where $\ell_0$ is the length of the simple
path $a\to b$ and $L$ the length of the cycle, then $W$ equals that
simple path preceded (or followed) by $m$ full cycles, however $W$ was
produced.  Likewise, a closed positive walk of length $L$ based at a
vertex $w$ of the cycle is letter-for-letter equal to $C(w)$.
\end{remark}

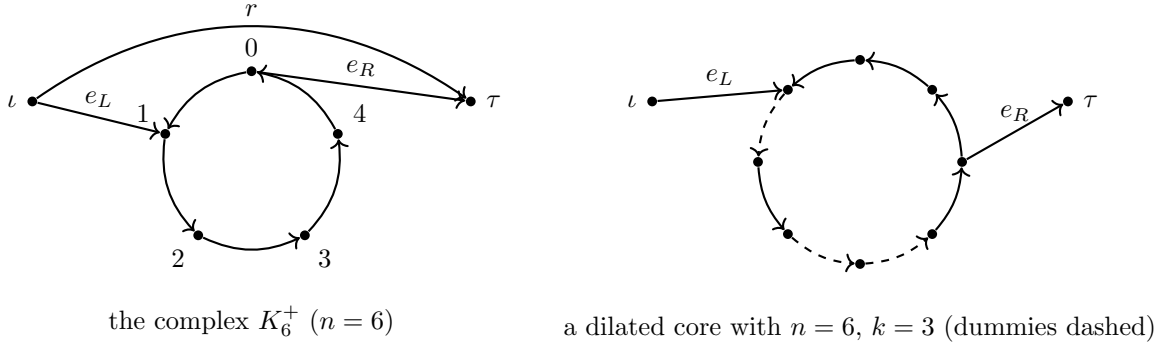
\begin{figure}[ht]
\centering
\begin{tikzpicture}[scale=1.0, every node/.style={font=\small},
    active/.style={->, thick}, dummy/.style={->, thick, dashed}]
  \begin{scope}
    \foreach \a/\i in {90/0, 162/1, 234/2, 306/3, 18/4} {
      \node[circle,fill,inner sep=1.3pt,label={\a:$\i$}] (v\i) at (\a:1.2) {};
    }
    \foreach \s/\t in {0/1,1/2,2/3,3/4,4/0} { \draw[active] (v\s) to[bend right=25] (v\t); }
    \node[circle,fill,inner sep=1.3pt,label=left:$\io$] (i) at (-2.9,0.8) {};
    \node[circle,fill,inner sep=1.3pt,label=right:$\ta$] (t) at (2.9,0.8) {};
    \draw[active] (i) -- (v1) node[midway,above] {$\eL$};
    \draw[active] (v0) -- (t) node[midway,above] {$\eR$};
    \draw[active] (i) to[bend left=35] node[midway,above] {$r$} (t);
    \node at (0,-2.1) {the complex $K_6^{+}$ ($n=6$)};
  \end{scope}
  \begin{scope}[xshift=8.05cm]
    \foreach \a/\i in {90/0, 135/1, 180/2, 225/3, 270/4, 315/5, 0/6, 45/7} {
      \node[circle,fill,inner sep=1.3pt] (w\i) at (\a:1.35) {};
    }
    \draw[active] (w0) to[bend right=18] (w1);
    \draw[dummy]  (w1) to[bend right=18] (w2);
    \draw[active] (w2) to[bend right=18] (w3);
    \draw[dummy]  (w3) to[bend right=18] (w4);
    \draw[dummy]  (w4) to[bend right=18] (w5);
    \draw[active] (w5) to[bend right=18] (w6);
    \draw[active] (w6) to[bend right=18] (w7);
    \draw[active] (w7) to[bend right=18] (w0);
    \node[circle,fill,inner sep=1.3pt,label=left:$\io$] (i2) at (-2.75,0.8) {};
    \node[circle,fill,inner sep=1.3pt,label=right:$\ta$] (t2) at (2.75,0.8) {};
    \draw[active] (i2) -- (w1) node[midway,above] {$\eL$};
    \draw[active] (w6) -- (t2) node[midway,above] {$\eR$};
    \node at (0,-2.2) {a dilated core with $n=6$, $k=3$ (dummies dashed)};
  \end{scope}
\end{tikzpicture}
\caption{The complex $K_6^{+}$ (left) and a dilated core of $F_6$
(right).
Active edges carry the relation ``once around the cycle''; dummy edges
carry no relation.}
\label{fig:cores}
\end{figure}

\begin{lemma}[Path equivalence in a dilated core]\label{lem:dilated-paths}
Let $K$ be a dilated core and let $p_1,p_2$ be $1$-paths with the same
initial and terminal vertices.  If neither is empty or consists entirely
of dummy edges, then $p_1$ and $p_2$ are homotopic.  In particular, any
two $1$-paths from $\io$ to $\ta$ are homotopic.
\end{lemma}

\begin{proof}
Write $p_i=Aq_iB$, where $A\in\{\eL,\varepsilon\}$,
$B\in\{\eR,\varepsilon\}$, and $q_1,q_2$ are inner paths with common
endpoints.  These factors are forced by the endpoints.  Assume
$|q_1|\geq|q_2|$, and let $w$ be their common terminal vertex.  Since the
inner $1$-skeleton is a directed simple cycle, the two paths differ by
full laps:
\[
q_1=q_2C(w)^m
\]
for some $m\geq0$.

For an inner edge $x\colon i\to j$, the word identity
$xC(j)=C(i)x$ from \cref{rem:cycle-convention} slides a full cycle past
$x$.  Since $p_2$ is not purely dummy, it contains an active inner edge,
$\eL$, or $\eR$.  Applying the positive cell of that edge inserts a full
cycle, and the word identity slides the cycle to the end of the inner
segment.  Repeating gives
\[
p_2\simeq Aq_2C(w)^mB=p_1.
\]
Every $\io$-to-$\ta$ path contains $\eL$, so the final assertion follows.
\end{proof}

\subsection{Sliding the boundary attachments}

\begin{proposition}[Attachment sliding]\label{prop:sliding}
Let $K$ be a dilated core of $F_n$ with $k$ dummy edges and attachments
$(u,v)$, and let $K'$ be identical except for attachments $(u',v')$.
Then $K$ and $K'$ are GS-equivalent.
\end{proposition}

\begin{proof}
It suffices to slide one attachment one step along the cycle.  We treat
the right attachment; the left attachment is symmetric.  If the inner
cycle has length $1$ there is nothing to prove.  Otherwise let $g$ be
the cycle edge leaving $v$, put $v^{+}=\ta(g)$, and let $P$ be the cycle
path from $v^{+}$ to $v$, so that
$C(v)=gP$ and $C(v^{+})=Pg$.  Adjoin, by (M1), a new edge
$\eR'\colon v^{+}\to\ta$ with defining positive cell
\[
\eR'=P\,\eR .
\]
The positive cell of $\eR$ reads $\eR=C(v)\,\eR=g\,P\,\eR$; by (M2), replace the
occurrence of $P\eR$ by $\eR'$, obtaining $\eR=g\,\eR'$.  Now apply
(M2) to the defining cell of $\eR'$, substituting $g\eR'$ for $\eR$:
\[
\eR'=P\,\eR \;\longrightarrow\; \eR'=P\,g\,\eR'=C(v^{+})\,\eR' .
\]
At this point $\eR$ occurs in exactly one $2$-cell, $\eR=g\eR'$, whose
other side avoids $\eR$; delete $\eR$ by (M1)$^{-1}$.  The result is
exactly the dilated core with attachments $(u,v^{+})$.  Note that no positive cell of the edge $g$ was used, so the argument is indifferent to
whether $g$ is active or dummy.
\end{proof}

\subsection{Contracting the dummy edges}

Dummy edges are the top edge of no $2$-cell, so a GS sequence can never
eliminate them --- at most exchange one for a fresh parallel edge with
the same endpoints (see the discussion after \cref{def:treelike});
contracting them is an operation of a different nature, and we now show
directly that it does not change the diagram group.  We first
normalize the attachments, using \cref{prop:sliding}; this makes the
combinatorics of paths in the dilated core transparent.

Let $K$ be a dilated core with $k\geq1$ dummy edges.  A \emph{chain} is a maximal non-empty $1$-path of consecutive dummy
edges on the inner cycle; since
$n\geq2$ there is at least one active edge, so the chains are pairwise
disjoint directed simple paths, and every dummy edge lies on exactly one
chain.  We say the attachments $(u,v)$ are \emph{adapted} if the cycle
edge entering $u$ is active and the cycle edge leaving $v$ is active;
equivalently, if $u$ is not an interior or terminal vertex of a chain,
and $v$ is not an interior or initial vertex of a chain.  Adapted
attachments exist, since there is at least one active edge.

\begin{lemma}[Projection Lemma]\label{lem:projection}
Let $K$ be a dilated core of $F_n$ with at least one dummy edge and
adapted attachments $(u,v)$.  Let $K^c$ be obtained by contracting every
dummy edge, and let $\pi$ delete dummy letters from $1$-paths.  Then
$K^c$ is a shifted core of $F_n$, and for every $1$-path $p$ from $\io$
to $\ta$,
\[
\DG(K,p)\cong\DG\bigl(K^c,\pi(p)\bigr).
\]
\end{lemma}

\begin{proof}
Contracting the dummy chains turns the inner cycle into the directed
simple cycle formed by the $n-1$ active inner edges.  Each positive cell
$f$ of $K$ contracts to a corresponding positive cell $f^c$ of $K^c$.
This gives a bijection between their positive cells and shows that $K^c$
is a shifted core.

Call a $1$-path in $K$ \emph{taut} if its first and last edges are
non-dummy.  Every maximal dummy subpath of a taut path is automatically
a complete chain.  Indeed, between two inner non-dummy edges this is
forced by the directed inner cycle; when one of the boundary edges
$\eL,\eR$ is involved, the same conclusion follows from adaptedness.

Every $\io$-to-$\ta$ path is taut: it has the form $\eL w\eR$, where
$w$ follows the inner cycle from $u$ to $v$.  Both paths of every
positive cell are taut as well.  This is immediate for the single-edge
path.  For the other paths
\[
eC(\ta(e)),\qquad \eL C(u),\qquad C(v)\eR,
\]
the first and last edges are non-dummy: for $eC(\ta(e))$ they are both
$e$, while for the other two paths this follows from adaptedness.

A taut path is uniquely determined by its sequence of non-dummy edges:
between two consecutive such edges the complete dummy chain is forced.
Deleting dummy edges and inserting these forced chains are inverse
operations.  Consequently, $\pi$ is a bijection between the
$\io$-to-$\ta$ paths of $K$ and those of $K^c$.

Let $\DG_{\io,\ta}(K)$ be the subgroupoid of $\DG(K)$ consisting of the
reduced diagrams whose top and bottom labels are $\io$-to-$\ta$ paths,
and define $\DG_{\io,\ta}(K^c)$ similarly.  Given a diagram $\Delta$ in
this subgroupoid, contract every edge labelled by a dummy edge and
replace each face labelled by $f^\epsilon$, where
$\epsilon\in\{1,-1\}$, by a face labelled by $(f^c)^\epsilon$.

We verify that the resulting object is a diagram over $K^c$.  No face
disappears, since both boundary paths of every cell contain a non-dummy
edge.  Contracting the dummy-labelled edges preserves planarity and the
unique source and sink.  It also creates no directed cycle: otherwise,
immediately before contracting some dummy edge $d$, there would be a
second directed path from $\io(d)$ to $\ta(d)$.  The region between that
path and $d$ would contain a non-trivial subdiagram with one boundary
path equal to $d$, which is impossible because no cell has a dummy edge
as an entire boundary path.

For each face labelled by $f^\epsilon$, its two boundary paths contract
to
\[
\pi\bigl(\toppath{f^\epsilon}\bigr)
   =\toppath{(f^c)^\epsilon},
\qquad
\pi\bigl(\botpath{f^\epsilon}\bigr)
   =\botpath{(f^c)^\epsilon}.
\]
Common boundary paths of incident faces are contracted identically, and
the outer boundary paths become
$\pi(\toppath\Delta)$ and $\pi(\botpath\Delta)$.  Hence the quotient is
a diagram over $K^c$; denote it by $\pi(\Delta)$.

Contraction preserves dipoles.  Conversely, suppose two faces of
$\pi(\Delta)$ form a dipole, and let $F_1,F_2$ be their corresponding
faces in $\Delta$.  Since contraction removes no faces, $F_1$ and $F_2$
are adjacent, and their labels are inverse cells.  Their facing boundary
paths have the same contraction.  Both paths are taut, so uniqueness of
taut expansion implies that they coincide.  Thus $F_1,F_2$ form a
dipole in $\Delta$.  In particular, reduced diagrams contract to
reduced diagrams.

Contraction therefore defines a map
\[
\Pi\colon \DG_{\io,\ta}(K)\longrightarrow
\DG_{\io,\ta}(K^c),\qquad
\Pi(\Delta)=\pi(\Delta).
\]
It respects the partially defined product and inverses: contraction
commutes with concatenation and reflection, and the preceding
equivalence shows that reduction commutes with contraction.  The map is
injective.  Indeed, if $\Pi(\Delta)$ is a trivial diagram, then it has no
cells; contraction does not delete cells, so $\Delta$ is trivial.

For surjectivity, we first record the required occurrence bijection.
Let $q$ be an $\io$-to-$\ta$ path in $K$, let $s$ be either path of a
cell, and put $s^c=\pi(s)$.  Then
\[
q=a s b\quad\longmapsto\quad
\pi(q)=\pi(a)s^c\pi(b)
\]
is a bijection between occurrences of $s$ in $q$ and occurrences of
$s^c$ in $\pi(q)$.  Indeed, the non-dummy edges determine the endpoints
of the occurrence, and all intervening dummy chains are forced.

Trivial diagrams over $K^c$ lift by the bijection on
$\io$-to-$\ta$ paths.  An atomic diagram over $K^c$ is determined by
its top path, a cell, and an occurrence of the relevant path of that
cell.  Each of these data has a unique preimage under the path, cell,
and occurrence bijections, and hence every atomic diagram lifts.

Every non-trivial diagram is a finite concatenation of atomic diagrams.
The atomic lifts are composable because the intermediate paths have
unique preimages.  Their concatenation is reduced: a dipole in the
lifted concatenation would contract to a dipole in the given reduced
diagram over $K^c$.  Thus every reduced diagram over $K^c$ lies in the
image of $\Pi$.

Hence
\[
\DG_{\io,\ta}(K)\cong\DG_{\io,\ta}(K^c).
\]
Taking the group with base path $p$ gives the asserted isomorphism.
\end{proof}

\subsection{The diagram group of a dilated core is \texorpdfstring{$F_n$}{Fn}}

\begin{corollary}\label{cor:dilated-Fn}
Let $K$ be a dilated core of $F_n$ \textup{(}any $k\geq0$, any placement
of the dummy edges, any attachments\textup{)} and let $p$ be any positive
path from $\io$ to $\ta$ in $K$.  Then $\DG(K,p)\cong F_n$.
\end{corollary}

\begin{proof}
By \cref{lem:dilated-paths} and \cref{lem:homotopic-base}, the group
$\DG(K,p)$ does not depend on the choice of the positive $\io\to\ta$ base
path $p$; we may therefore change base paths freely below.

If $k\geq1$, slide the attachments (\cref{prop:sliding}) to adapted ones;
by \cref{lem:transport} this does not change the diagram group, and the
transported base is again a positive $\io\to\ta$ path (the edge mapping
preserves endpoints, and $\io,\ta$ are fixed).  By the Projection Lemma
(\cref{lem:projection}) the diagram group equals that of a shifted core
of $F_n$.  So we may assume $k=0$.

By \cref{prop:sliding} again, a shifted core with arbitrary attachments
is GS-equivalent to the standard core $K_n$; by \cref{lem:transport} the
diagram groups agree (over transported bases).  Finally
$\DG(K_n,q)\cong F_n$ for any positive $\io\to\ta$ base $q$ by
\cref{cor:standard-core}, \cref{lem:dilated-paths} and
\cref{lem:homotopic-base}.
\end{proof}

\section{The directed \texorpdfstring{$2$}{2}-complex of a dynamical diagram}\label{sec:complex}

In this section we extend the semigroup presentation of Belk and Stott
from canonical partitions to the more general partitions needed in the
induction, and realize it as a directed $2$-complex $\Kdyn(B)$ with only
the vertex identifications forced by its cells.  We then determine the
vertices of its $1$-skeleton and prove that its inner vertices form a
strongly connected component.

\subsection{Partitions and the dynamical presentation}

\begin{definition}[General and canonical partitions]\label{def:partition}
Let $B$ be a fast set of $n$ positive bumps with a marking.  A
\emph{general partition} associated with $B$ is a finite sequence
\[
A_1<A_2<\dots<A_N
\]
of pairwise disjoint intervals that contain the $2n$ feet of $B$ and
whose union agrees with $\supp(B)$ up to finitely many points.  The
non-foot intervals are called \emph{dummy intervals}; write
$N=2n+k$, where $k\geq0$ is their number.  The first interval is the
leftmost foot of $B$, and the last is the rightmost foot.

Suppose $B$ has $s$ isolated bumps.  A \emph{canonical partition} is a
general partition whose only non-foot intervals are, for each isolated
bump, one fundamental-domain interval between its two feet.  Thus a
canonical partition has $2n+s$ intervals.  A marking that gives such a
partition is a canonical configuration in the sense of
\cref{def:canonical-config}.
\end{definition}

Belk and Stott use canonical partitions \cite[\S3.1]{BS}.  The broader
notion above allows additional dummy intervals; this extra freedom is
needed when partitions are truncated in \cref{sec:reduction}.

For $b\in B$, let $i(b)$ and $j(b)$ be determined by
$A_{i(b)}=L_b$ and $A_{j(b)}=R_b$, and set
\[
\gap b=A_{i(b)+1}A_{i(b)+2}\cdots A_{j(b)-1}.
\]
We call $\gap b$ the \emph{partition gap path} of $b$.  A foot of
another bump that meets $\supp(b)$ lies in the gap of $b$, since it is
disjoint from both feet of $b$.  Hence the feet occurring in the
partition gap path $\gap b$ are precisely the feet in the combinatorial
gap word of $b$; the dummy intervals refine that word.  The intervals in
$\gap b$ cover the fundamental domain $[m_b,b(m_b))$ up to finitely many
points.  This motivates the following extension of the Belk--Stott
presentation.

\begin{definition}[The dynamical presentation]\label{def:presB}
The \emph{dynamical presentation} of $(B,A_1\cdots A_N)$ is the indexed
semigroup presentation
\[
\Pres_B=\bigl\langle A_1,\dots,A_N\bigm|
\lambda_b,\rho_b\ (b\in B)\bigr\rangle,
\]
where the positive rules are
\[
\lambda_b\colon A_{i(b)}\longrightarrow A_{i(b)}\gap b,
\qquad
\rho_b\colon A_{j(b)}\longrightarrow \gap b\,A_{j(b)}.
\]
\end{definition}

\begin{theorem}[Belk--Stott; {\cite[Theorem~3.7]{BS}}]\label{thm:BS}
Let $B$ be a geometrically fast set of $n$ positive bumps, let $s$ be the
number of isolated bumps, and let $A_1,\dots,A_{2n+s}$ be its canonical
partition.  Then
\[
\langle B\rangle\cong
\DG\bigl(\Pres_B,A_1A_2\cdots A_{2n+s}\bigr).
\]
\end{theorem}

\subsection{The complex \texorpdfstring{$\Kdyn$}{Kdyn}}

The presentation $\Pres_B$ has base path $A_1\cdots A_N$.  We give
these letters the vertex structure obtained from the base line graph by
imposing the identifications forced by the cells.

Making $A_1\cdots A_N$ a $1$-path first forces the following line graph.

\begin{definition}[Base line graph]\label{def:base-line}
The \emph{base line graph} $\Gamma_{\mathrm{base}}$ has vertices
$v_1,\dots,v_{N+1}$, edges $A_1,\dots,A_N$, and incidence maps
\[
\io(A_m)=v_m,\qquad \ta(A_m)=v_{m+1}.
\]
Every contiguous subword of the base word is then a $1$-path.
\end{definition}

\begin{lemma}[Gap strictness]\label{lem:gap-strict}
For every $b\in B$, one has $j(b)-i(b)\geq2$; equivalently, $\gap b$ is
non-empty.
\end{lemma}

\begin{proof}
The fundamental domain $[m_b,b(m_b))$ has positive length and is covered,
up to finitely many points, by partition intervals strictly between the
two feet of $b$.
\end{proof}

\begin{proposition}[Forced identifications]\label{prop:identification}
For each $b\in B$, both positive cells $\lambda_b$ and $\rho_b$ force the
same identification
\[
v_{i(b)+1}\sim v_{j(b)},
\]
and no other identification.
\end{proposition}

\begin{proof}
The paths of $\lambda_b$ begin with the same edge $A_{i(b)}$.  Their
terminal vertices are $v_{i(b)+1}$ and $v_{j(b)}$, so the displayed
identification is necessary and sufficient.  The paths of $\rho_b$ end
with the same edge $A_{j(b)}$, while their initial vertices are
$v_{j(b)}$ and $v_{i(b)+1}$, giving the same identification.
\end{proof}

\begin{definition}[The dynamical complex $\Kdyn$]\label{def:Kdyn}
Let $\sim$ be the equivalence relation on the base vertices generated by
$v_{i(b)+1}\sim v_{j(b)}$ for $b\in B$.  The \emph{dynamical complex}
$\Kdyn(B)$---more precisely, the complex associated with
$(B,A_1\cdots A_N)$---has $1$-skeleton
\[
\Gamma^{(1)}=\Gamma_{\mathrm{base}}/\!\sim
\]
and positive cells $\lambda_b,\rho_b$ for $b\in B$; see
\cref{fig:skeleton}.  It is tree-like: the feet are its $2n$ active
edges, and the $k$ non-foot intervals are its dummy edges.  For a base
vertex $v$, write $[v]$ for its $\sim$-class.  When no confusion can
arise, any representative may be used inside the brackets.
\end{definition}

\begin{proposition}[Admissibility]\label{prop:admissible}
Every positive cell of $\Pres_B$ is valid over $\Gamma^{(1)}$, so
$\Kdyn(B)$ is a directed $2$-complex.  Hence, for every $1$-path $q$ of
$\Kdyn(B)$,
\[
\DG\bigl(\Kdyn(B),q\bigr)\cong\DG(\Pres_B,q).
\]
\end{proposition}

\begin{proof}
\Cref{prop:identification} gives the identifications needed to make both
paths of every cell have common endpoints.  The diagram-group
isomorphism follows from \cref{lem:vertices} after identifying all
vertices.
\end{proof}

\begin{theorem}[Representation; {\cite[Theorem~3.7]{BS}}]
\label{thm:representation}
Let $B$ be a geometrically fast set of $n$ positive bumps, let $s$ be the
number of isolated bumps, and let $A_1,\dots,A_{2n+s}$ be its canonical
partition.  Then
\[
\langle B\rangle\cong
\DG\bigl(\Kdyn(B),A_1A_2\cdots A_{2n+s}\bigr).
\]
\end{theorem}

\begin{proof}
Combine \cref{thm:BS,prop:admissible}.
\end{proof}

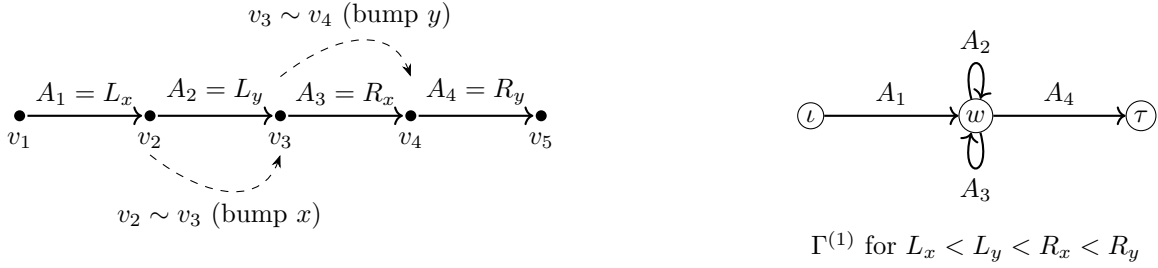
\begin{figure}[ht]
\centering
\begin{tikzpicture}[scale=1.15, every node/.style={font=\small}]
  \begin{scope}
    \foreach \i in {1,...,5} { \fill (\i*1.5-1.5,0) circle (1.8pt); \node[below] at (\i*1.5-1.5,-0.08) {$v_{\i}$}; }
    \foreach \i/\l in {1/{A_1=L_x},2/{A_2=L_y},3/{A_3=R_x},4/{A_4=R_y}} {
      \draw[->,thick] (\i*1.5-1.5+0.09,0) -- (\i*1.5-0.09,0) node[midway,above] {$\l$};
    }
    \draw[dashed,-{Stealth}] (1.5,-0.45) .. controls (2.2,-1.0) and (2.8,-1.0) .. (3.0,-0.45);
    \node at (2.3,-1.15) {$v_2\sim v_3$ (bump $x$)};
    \draw[dashed,-{Stealth}] (3.0,0.45) .. controls (3.7,1.0) and (4.3,1.0) .. (4.5,0.45);
    \node at (3.8,1.15) {$v_3\sim v_4$ (bump $y$)};
  \end{scope}
  \begin{scope}[xshift=9.1cm,yshift=0cm]
    \node[circle,draw,inner sep=1.5pt] (i) at (0,0) {$\io$};
    \node[circle,draw,inner sep=1.5pt] (w) at (1.9,0) {$w$};
    \node[circle,draw,inner sep=1.5pt] (t) at (3.8,0) {$\ta$};
    \draw[->,thick] (i) -- (w) node[midway,above] {$A_1$};
    \draw[->,thick] (w) -- (t) node[midway,above] {$A_4$};
    \draw[->,thick] (w) edge[loop above, looseness=14] node[above] {$A_2$} (w);
    \draw[->,thick] (w) edge[loop below, looseness=14] node[below] {$A_3$} (w);
    \node at (1.9,-1.5) {$\Gamma^{(1)}$ for $L_x<L_y<R_x<R_y$};
  \end{scope}
\end{tikzpicture}
\caption{The base line graph of the canonical two-bump configuration
generating $F$, the identifications forced by the two bumps, and the
resulting $1$-skeleton, in which $w=\{v_2,v_3,v_4\}$ and the inner feet
$A_2,A_3$ become loops.}
\label{fig:skeleton}
\end{figure}

\subsection{The structure of the \texorpdfstring{$1$}{1}-skeleton}

\begin{lemma}[Vertex collapse and classification]\label{lem:vertex-collapse}
Let $B$ be a fast set of $n\geq1$ positive bumps with a general partition
of $N=2n+k$ intervals.  Then $\Gamma^{(1)}$ has exactly $n+k+1$ vertices:
\begin{enumerate}
\item the initial vertex $\io=[v_1]$, with no incoming edges and with
$A_1$ as its unique outgoing edge;
\item the terminal vertex $\ta=[v_{N+1}]$, with no outgoing edges and
with $A_N$ as its unique incoming edge;
\item exactly $n+k-1$ inner vertices, each with at least one incoming and
one outgoing edge.
\end{enumerate}
\end{lemma}

\begin{proof}
Order the bumps so that
$j(b_1)<\cdots<j(b_n)$.  When the identification
$v_{i(b_m)+1}\sim v_{j(b_m)}$ is applied, the vertex $v_{j(b_m)}$ has not
occurred in an earlier identification: for $r<m$,
$i(b_r)+1\leq j(b_r)<j(b_m)$.  Each of the $n$ identifications therefore
merges two classes, and the number of classes drops from
$N+1=2n+k+1$ to $n+k+1$.

The vertices $v_1$ and $v_{N+1}$ are never identified with another base
vertex, giving the first two cases.  Every other class contains a base
vertex with one incoming and one outgoing edge, so it has both.
\end{proof}

\begin{proposition}[Inner strong connectivity]\label{prop:inner-connected}
Let $B$ be an irreducible fast set of $n\geq2$ positive bumps with a
general partition.  The inner vertices of $\Kdyn(B)$ form a strongly
connected component.
\end{proposition}

\begin{proof}
For $2\leq q<q'\leq N$, the images of
$A_q,A_{q+1},\dots,A_{q'-1}$ give a path from $[v_q]$ to $[v_{q'}]$.
It therefore suffices to show that every inner vertex reaches $[v_2]$.

Fix $[v_m]$, $2\leq m\leq N$, and let
\[
\mathcal R=\{q\in\{2,\dots,N\}\mid [v_q]\text{ is reachable from }
[v_m]\}.
\]
Then $m\in\mathcal R$.  If $q\in\mathcal R$ and $q<N$, then
$q+1\in\mathcal R$; and if $v_q\sim v_{q'}$, then
$q\in\mathcal R$ if and only if $q'\in\mathcal R$.  Let
$\mu=\min\mathcal R$.  Then every index from $\mu$ through $N$ lies in
$\mathcal R$.

Suppose $\mu\geq3$.  If $j(b)\geq\mu$, then
$j(b)\in\mathcal R$ and the identification
$v_{i(b)+1}\sim v_{j(b)}$ gives $i(b)+1\in\mathcal R$.  Hence
$i(b)\geq\mu-1$.  Set
\[
B_{\mathrm{left}}=\{b\mid j(b)\leq\mu-1\},\qquad
B_{\mathrm{right}}=\{b\mid j(b)\geq\mu\}.
\]
The owner of $A_N$ lies in $B_{\mathrm{right}}$.  The owner $b^*$ of
$A_1$ lies in $B_{\mathrm{left}}$: otherwise $j(b^*)\geq\mu$ would imply
$i(b^*)\geq\mu-1$, contradicting $i(b^*)=1$.  Thus both sets are
non-empty.  For $b\in B_{\mathrm{left}}$ and
$b'\in B_{\mathrm{right}}$,
\[
j(b)\leq\mu-1\leq i(b'),
\]
so $b$ and $b'$ cannot cross.  This disconnects the crossing graph,
contrary to irreducibility.  Therefore $\mu=2$, and every inner vertex
reaches $[v_2]$.  Since $[v_2]$ reaches every inner vertex by walking to
the right, the inner vertices are strongly connected.  The source and
sink do not belong to this component because the source cannot be
entered and the sink cannot be left.
\end{proof}

\section{Reduction of the dynamical complex to a dilated core}
\label{sec:reduction}

Throughout this section, let $B=\{b_1,\dots,b_n\}$ be a peelable fast
set of $n\geq2$ positive bumps, indexed so that
\[
R_1<R_2<\dots<R_n,
\]
and let $A_1,\dots,A_N$, $N=2n+k$, be a general partition.  Then $A_1$
is the left foot of the leftmost bump and $A_N=R_n$.

\subsection{Lifting GS sequences}

The inductive step deletes the last bump, applies the induction
hypothesis to the smaller complex, and replays the resulting GS sequence
on the larger complex.  The next lemma justifies that replay.

\begin{lemma}[Lifting Lemma]\label{lem:lifting}
Let $K=(V,E,\mathcal R)$ be a directed $2$-complex, where $\mathcal R$
is its set of positive cells.  Let
\[
\widehat K=(\widehat V,E\sqcup F,\mathcal R\sqcup\mathcal S)
\]
be a directed $2$-complex together with a map
$\theta\colon V\to\widehat V$ such that the endpoints in $\widehat K$ of
each edge of $E$ are the $\theta$-images of its endpoints in $K$.  Let
$\Sigma$ be a GS sequence from $K$ to
$K^*=(V,E^*,\mathcal R^*)$, with fresh edges disjoint from $F$.  Then
$\Sigma$ lifts to a GS sequence $\widehat\Sigma$ from $\widehat K$ to
\[
\widehat K^*=(\widehat V,E^*\sqcup F,
\mathcal R^*\sqcup\mathcal S^*),
\]
where the endpoints of the edges of $E^*$ are the $\theta$-images of
their endpoints in $K^*$, and $\mathcal S^*$ is obtained by applying
$\phi_\Sigma$ to both paths of every cell of $\mathcal S$, extended as
the identity on $F$.  Moreover, $\phi_{\widehat\Sigma}$ agrees with
$\phi_\Sigma$ on $E$ and fixes $F$.
\end{lemma}

\begin{proof}
Replay $\Sigma$ move by move.  After an initial segment has transformed
$K$ into $(V,E_i,\mathcal R_i)$, maintain the complex
\[
(\widehat V,E_i\sqcup F,\mathcal R_i\sqcup\mathcal S_i),
\]
where the endpoints of edges in $E_i$ are mapped by $\theta$ and
$\mathcal S_i$ is obtained by the accumulated edge substitutions.

An (M2) move rewrites one cell of $\mathcal R_i$ using another; both are
present upstairs, so the same move is valid.  An (M1) move adjoining
$x\to w$ is replayed with endpoints
$\theta(\io(w)),\theta(\ta(w))$.  If (M1)$^{-1}$ deletes $x$ with cell
$x\to w$, then $x$ may still occur in cells of $\mathcal S_i$.  First
use (M2) to replace every such occurrence by $w$, and then delete $x$.
This is exactly the substitution prescribed by the edge mapping.  The
claimed final complex and edge mapping follow.
\end{proof}

\subsection{Recognizing dilated cores}

\begin{lemma}[Cycle Recognition Lemma]\label{lem:cycle-rec}
Let $K$ have vertex set $\{\io,\ta\}\sqcup\Vin$, with $\io\neq\ta$, and
suppose:
\begin{enumerate}
\item $\io$ has no incoming edges and $\ta$ has no outgoing edges;
\item the vertices of $\Vin$ form a strongly connected component;
\item the number of inner edges, those with both endpoints in $\Vin$,
is $|\Vin|=m\geq1$.
\end{enumerate}
Then the inner edges form a directed simple cycle through every vertex of
$\Vin$; for $m=1$ this is a single loop.
\end{lemma}

\begin{proof}
A path between inner vertices cannot pass through $\io$ or $\ta$, so the
inner subgraph is strongly connected.  Every inner vertex therefore has
at least one incoming and one outgoing inner edge.  Since the sum of the
inner in-degrees, and also of the inner out-degrees, is the number
$m$ of inner edges, every inner vertex has exactly one of each.  The
inner subgraph is thus a disjoint union of directed cycles, and strong
connectivity forces a single cycle.
\end{proof}

\begin{definition}[Admissible core]\label{def:admissible-core}
An \emph{admissible core} for $(B,A_1\cdots A_N)$ is a dilated core of
$F_n$ with $k$ dummy edges whose vertex set, source $\io$, and sink
$\ta$ are those of $\Kdyn(B)$; whose dummy edges are exactly the dummy
edges of $\Kdyn(B)$ with their original endpoints; and whose left and
right edges are $\eL=A_1$ and $\eR=A_N=R_n$.  Thus the left attachment
is $\ta(A_1)$ and the right attachment is $\io(R_n)$.  The active inner
edges are not prescribed as particular original edges and may include
fresh edges introduced by the GS sequence.
\end{definition}

\subsection{The Reduction Theorem}

\begin{theorem}[Reduction to a dilated core]\label{thm:reduction}
Let $B$ be a peelable fast set of $n\geq2$ positive bumps with a general
partition $A_1,\dots,A_N$.  There is a finite GS sequence transforming
$\Kdyn(B)$ into an admissible core for $(B,A_1\cdots A_N)$.  The sequence
never deletes $A_1$, $A_N=R_n$, or any dummy edge of the original
complex.
\end{theorem}

We use the following consequence of the structural results in
\cref{sec:complex}.

\begin{lemma}\label{lem:skeleton-invariants}
Let $\Sigma$ be any GS sequence beginning at $\Kdyn(B)$, where $B$ is
irreducible with at least two bumps, and let $K'$ be the resulting
complex.  The vertex set is unchanged, $\io$ has no incoming edge,
$\ta$ has no outgoing edge, and the inner vertices form a strongly
connected component of $K'$.
\end{lemma}

\begin{proof}
Combine \cref{lem:GSinv,lem:vertex-collapse,prop:inner-connected}.
\end{proof}

\subsection{The base case \texorpdfstring{$n=2$}{n=2}}

\begin{proof}[Proof of \cref{thm:reduction}, base case $n=2$]
Write $B=\{x,y\}$.  Peelability means that the two bumps cross, so, with
$x$ leftmost, the partition has the form
\[
e_1D_1e_2D_2e_3D_3e_4,
\qquad
e_1=L_x,\ e_2=L_y,\ e_3=R_x,\ e_4=R_y,
\]
where the $D_i$ are possibly empty paths of dummy edges.  The positive
cells are
\[
e_1=e_1D_1e_2D_2,\qquad e_3=D_1e_2D_2e_3,
\]
\[
e_2=e_2D_2e_3D_3,\qquad e_4=D_2e_3D_3e_4.
\]
The two forced identifications are
$\ta(e_1)\sim\io(e_3)$ and $\ta(e_2)\sim\io(e_4)$.  There are $k+1$
inner vertices, where $k$ is the total number of dummy edges.

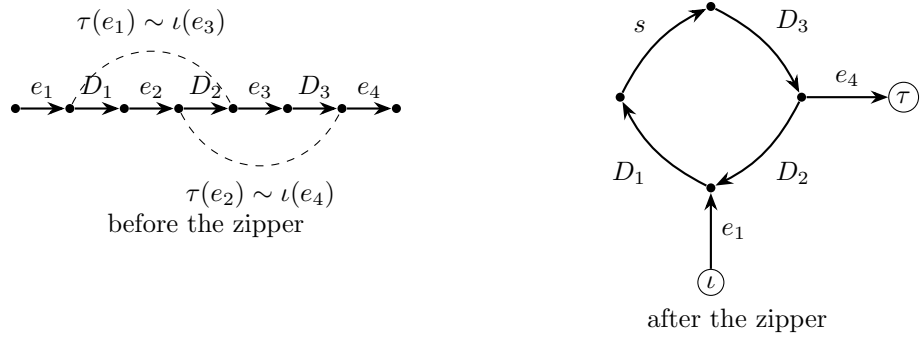
\begin{figure}[ht]
\centering
\begin{tikzpicture}[every node/.style={font=\small},>=Stealth]
  \begin{scope}[xscale=.72]
    \foreach \x/\name in {0/a,1/b,2/c,3/d,4/e,5/f,6/g,7/h} {
      \node[circle,fill,inner sep=1.2pt] (\name) at (\x,0) {};
    }
    \draw[->,thick] (a)--(b) node[midway,above] {$e_1$};
    \draw[->,thick] (b)--(c) node[midway,above] {$D_1$};
    \draw[->,thick] (c)--(d) node[midway,above] {$e_2$};
    \draw[->,thick] (d)--(e) node[midway,above] {$D_2$};
    \draw[->,thick] (e)--(f) node[midway,above] {$e_3$};
    \draw[->,thick] (f)--(g) node[midway,above] {$D_3$};
    \draw[->,thick] (g)--(h) node[midway,above] {$e_4$};
    \draw[dashed] (b) to[bend left=55]
      node[midway,above=2pt] {$\ta(e_1)\sim\io(e_3)$} (e);
    \draw[dashed] (d) to[bend right=55]
      node[midway,below=2pt] {$\ta(e_2)\sim\io(e_4)$} (g);
    \node at (3.5,-1.55) {before the zipper};
  \end{scope}
  \begin{scope}[xshift=9.2cm,yshift=.15cm]
    \node[circle,fill,inner sep=1.2pt] (A) at (0,1.2) {};
    \node[circle,fill,inner sep=1.2pt] (B) at (1.2,0) {};
    \node[circle,fill,inner sep=1.2pt] (C) at (0,-1.2) {};
    \node[circle,fill,inner sep=1.2pt] (D) at (-1.2,0) {};
    \node[circle,draw,inner sep=1.5pt] (I) at (0,-2.45) {$\io$};
    \node[circle,draw,inner sep=1.5pt] (T) at (2.55,0) {$\ta$};
    \draw[->,thick] (A) to[bend left=18]
      node[midway,above right] {$D_3$} (B);
    \draw[->,thick] (B) to[bend left=18]
      node[midway,below right] {$D_2$} (C);
    \draw[->,thick] (C) to[bend left=18]
      node[midway,below left] {$D_1$} (D);
    \draw[->,thick] (D) to[bend left=18]
      node[midway,above left] {$s$} (A);
    \draw[->,thick] (I)--(C) node[midway,right] {$e_1$};
    \draw[->,thick] (B)--(T) node[midway,above] {$e_4$};
    \node at (.35,-2.95) {after the zipper};
  \end{scope}
\end{tikzpicture}
\caption{The two-bump base case.  Each $D_i$ denotes a possibly empty
path of dummy edges.  In the left panel, the undirected dashed curves
indicate the forced vertex identifications
$\ta(e_1)\sim\io(e_3)$ and $\ta(e_2)\sim\io(e_4)$.  After adjoining the
zipper edge $s$ and deleting $e_2,e_3$, the right panel is a directed
cycle with cyclic word $D_3D_2D_1s$.}
\label{fig:zipper}
\end{figure}

Adjoin a fresh edge
\[
s\colon\io(e_2)\longrightarrow\ta(e_3)
\]
with positive cell $s=e_2D_2e_3$.  Using this cell in (M2), rewrite
\[
e_3=D_1(e_2D_2e_3)\longrightarrow e_3=D_1s,
\qquad
e_2=(e_2D_2e_3)D_3\longrightarrow e_2=sD_3.
\]
Substituting these paths into the cell of $s$ and the cells of $e_1,e_4$
gives
\[
s=s(D_3D_2D_1)s,
\]
\[
e_1=e_1D_1sD_3D_2,
\qquad
e_4=D_2D_1sD_3e_4.
\]
Now delete $e_2$ and $e_3$ by two moves (M1)$^{-1}$, and call the result
$K^*$.

The inner edges of $K^*$ are $s$ and the $k$ dummy edges, so their number
is the number $k+1$ of inner vertices.  By
\cref{lem:skeleton-invariants,lem:cycle-rec}, the inner edges form a
directed simple cycle.  The relation $s=s(D_3D_2D_1)s$ shows that
$D_3D_2D_1s$ is a closed path based at $\ta(s)$; see also
\cref{fig:zipper}.  Since it traverses every inner edge once, it is the
full cycle $C(\ta(s))$.  The paths $D_1sD_3D_2$ and $D_2D_1sD_3$ are the
corresponding cyclic rotations based at $\ta(e_1)$ and $\io(e_4)$.
Consequently the three remaining positive cells are
\[
s=sC(\ta(s)),\qquad
e_1=e_1C(\ta(e_1)),\qquad
e_4=C(\io(e_4))e_4.
\]
Thus $K^*$ is an admissible core.  Only $e_2$ and $e_3$ were deleted, so
$A_1=e_1$, $A_N=e_4$, and every dummy edge survives.
\end{proof}

\subsection{The inductive step}

\begin{proof}[Proof of \cref{thm:reduction}, inductive step]
Assume the theorem for peelable sets of $n$ bumps, and let
$B=\{b_1,\dots,b_{n+1}\}$ be peelable, with
$R_1<\cdots<R_{n+1}$ and a general partition $\mathcal A$ having $k$
dummy intervals.

\emph{Truncation.}
No left foot lies to the right of $R_n$.  This is clear for
$b_1,\dots,b_n$, while $b_{n+1}$ crosses an earlier bump by
\cref{lem:recognize}, so $L_{n+1}<R_j\leq R_n$ for some $j\leq n$.
Hence $\mathcal A$ ends with
\[
\mathcal A=\cdots R_n d_1d_2\cdots d_\ell R_{n+1}.
\]
Here $\ell\geq0$ and the $d_i$ are dummy intervals.  Let
$B_n=B\setminus\{b_{n+1}\}$ and form a truncated general partition
$\mathcal A_n$ by discarding $d_1,\dots,d_\ell,R_{n+1}$ and regarding
the interval $L_{n+1}$ as a dummy interval.  This is again a general
partition.  Indeed, if $b_{n+1}$ were also the leftmost bump, its span
would contain every other span, so it would cross no earlier bump,
contrary to peelability.  The truncated partition has
\[
k_n=k-\ell+1
\]
dummy intervals.  Write
\[
K_{\mathrm{tr}}=\Kdyn(B_n;\mathcal A_n),\qquad
K_{\mathrm{full}}=\Kdyn(B;\mathcal A).
\]

The partition gap paths of $b_1,\dots,b_n$ and their forced
identifications are the same in the two partitions.  Write
$K_{\mathrm{tr}}=(V,E,\mathcal R)$ and
\[
K_{\mathrm{full}}=(\widehat V,E\sqcup F,
\mathcal R\sqcup\{\lambda_{n+1},\rho_{n+1}\}),
\qquad
F=\{d_1,\dots,d_\ell,R_{n+1}\}.
\]
Let
\[
v_R=\io(R_n),\qquad
w=\ta(L_{n+1}),\qquad
\ta_n=\ta(R_n),\qquad
D=d_1\cdots d_\ell.
\]
The bump $b_{n+1}$ identifies the terminal vertex of $D$ with $w$.
There are two cases, illustrated in \cref{fig:truncation}.

If $\ell>0$, the path $D$ runs from $\ta_n$ to $w$; the larger vertex
set adds the $\ell-1$ interior vertices of $D$ and the new terminal
vertex $\ta_{n+1}=\ta(R_{n+1})$, and the map
$\theta\colon V\to\widehat V$ is the inclusion.  If $\ell=0$, then
$D$ is empty, $\ta_n$ is identified with $w$, and $\widehat V$ is the
quotient of $V$ by this identification together with the new terminal
vertex $\ta_{n+1}$.  Set
\[
t:=\theta(\ta_n)=
\begin{cases}
\ta_n,&\ell>0,\\
w,&\ell=0.
\end{cases}
\]
Thus $D$ is a path from $t$ to $w$ (the empty path when $\ell=0$), and
$R_{n+1}$ runs from $w$ to $\ta_{n+1}$.

\begin{figure}[ht]
\centering
\begin{tikzpicture}[every node/.style={font=\small},>=Stealth]
  \begin{scope}
    \node[circle,fill,inner sep=1.2pt,label=below:$v_R$] (vr) at (0,0) {};
    \node[circle,fill,inner sep=1.2pt,label=below:$t$] (tn) at (1.8,0) {};
    \node[circle,fill,inner sep=1.2pt,label=below left:$w$] (w) at (4.2,0) {};
    \node[circle,draw,inner sep=1.5pt] (tn1) at (6.2,0) {$\ta_{n+1}$};
    \node[circle,fill,inner sep=1.2pt,label=below:$z$] (z) at (4.2,-1.35) {};
    \draw[->,thick] (vr)--(tn) node[midway,above] {$R_n$};
    \draw[->,thick] (tn)--(w) node[midway,above] {$D=d_1\cdots d_\ell$};
    \draw[->,thick] (w)--(tn1) node[midway,above] {$R_{n+1}$};
    \draw[->,thick] (z)--(w) node[midway,right] {$L_{n+1}$};
    \node at (3.1,-2.0) {$\ell>0$};
  \end{scope}
  \begin{scope}[xshift=8.3cm]
    \node[circle,fill,inner sep=1.2pt,label=below:$v_R$] (vr0) at (0,0) {};
    \node[circle,fill,inner sep=1.2pt,label=above:{$w=t$}] (tw) at (2.1,0) {};
    \node[circle,draw,inner sep=1.5pt] (tn10) at (4.45,0) {$\ta_{n+1}$};
    \node[circle,fill,inner sep=1.2pt,label=below:$z$] (z0) at (2.1,-1.35) {};
    \draw[->,thick] (vr0)--(tw) node[midway,above] {$R_n$};
    \draw[->,thick] (tw)--(tn10) node[midway,above] {$R_{n+1}$};
    \draw[->,thick] (z0)--(tw) node[midway,right] {$L_{n+1}$};
    \node at (2.2,-2.0) {$\ell=0$};
  \end{scope}
\end{tikzpicture}
\caption{The two vertex configurations in the truncation step.  When
$\ell>0$, the discarded dummy path $D$ joins $t=\ta_n$ to $w$; when
$\ell=0$, the vertices satisfy $t=\ta_n=w$.}
\label{fig:truncation}
\end{figure}
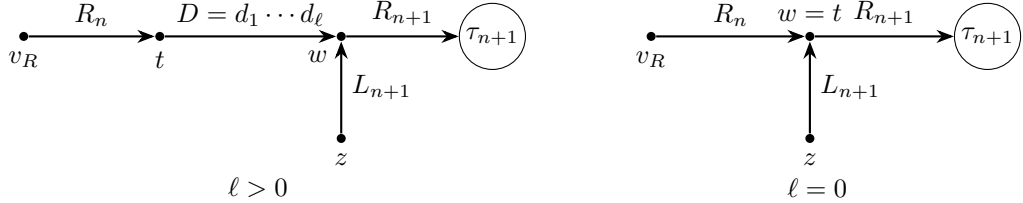

Thus $K_{\mathrm{full}}$ is related to $K_{\mathrm{tr}}$ exactly as
the larger complex in the Lifting Lemma is related to the smaller one.

\emph{Applying the induction hypothesis and lifting.}
Let $\Sigma$ transform $K_{\mathrm{tr}}$ into an admissible core
$K_{\mathrm{core}}=(V,E^*,\mathcal R^*)$.  Its inner cycle contains
the $k_n$ dummy edges of the truncated partition, including
$L_{n+1}$, and $n-1$ active edges.  Its left and right edges are
$A_1$ and $R_n$.  Since $\Sigma$ never deletes $A_1,R_n$, or a dummy
edge, its edge mapping fixes these edges, in particular $L_{n+1}$.

Lift $\Sigma$ to $K_{\mathrm{full}}$ by \cref{lem:lifting}.  The
resulting mixed complex has the mapped copy of $K_{\mathrm{core}}$ together
with the new edges $F$ and two lifted cells
$\lambda_{n+1}^*,\rho_{n+1}^*$.

Write $L_{n+1}\colon z\to w$.  Since $L_{n+1}$ is a dummy edge on the
old inner cycle, write
\[
C(v_R)=X L_{n+1}Y,
\]
where $X$ is the simple cycle path from $v_R$ to $z$ and $Y$ is the
simple cycle path from $w$ to $v_R$.  The two forms of the mixed complex
are shown in \cref{fig:mixed-complex}.

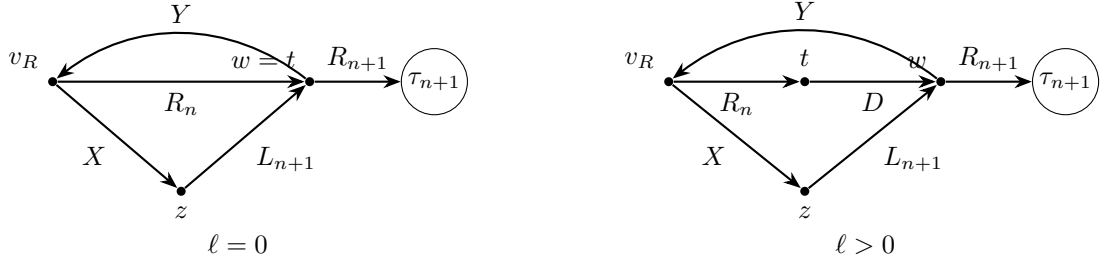
\begin{figure}[ht]
\centering
\begin{tikzpicture}[every node/.style={font=\small},>=Stealth]
  \begin{scope}
    \node[circle,fill,inner sep=1.2pt,label=above left:$v_R$] (vr0) at (-1.7,0) {};
    \node[circle,fill,inner sep=1.2pt,label=below:$z$] (z0) at (0,-1.45) {};
    \node[circle,fill,inner sep=1.2pt,label=above left:{$w=t$}] (w0) at (1.7,0) {};
    \node[circle,draw,inner sep=1.5pt] (tn10) at (3.35,0) {$\ta_{n+1}$};
    \draw[->,thick] (vr0)--(z0) node[midway,below left] {$X$};
    \draw[->,thick] (z0)--(w0) node[midway,below right] {$L_{n+1}$};
    \draw[->,thick] (w0) to[bend right=38] node[midway,above] {$Y$} (vr0);
    \draw[->,thick] (vr0)--(w0) node[midway,below] {$R_n$};
    \draw[->,thick] (w0)--(tn10) node[midway,above] {$R_{n+1}$};
    \node at (.75,-2.15) {$\ell=0$};
  \end{scope}
  \begin{scope}[xshift=8.25cm]
    \node[circle,fill,inner sep=1.2pt,label=above left:$v_R$] (vrp) at (-1.8,0) {};
    \node[circle,fill,inner sep=1.2pt,label=above:$t$] (tp) at (0,0) {};
    \node[circle,fill,inner sep=1.2pt,label=below:$z$] (zp) at (0,-1.45) {};
    \node[circle,fill,inner sep=1.2pt,label=above left:$w$] (wp) at (1.8,0) {};
    \node[circle,draw,inner sep=1.5pt] (tn1p) at (3.45,0) {$\ta_{n+1}$};
    \draw[->,thick] (vrp)--(zp) node[midway,below left] {$X$};
    \draw[->,thick] (zp)--(wp) node[midway,below right] {$L_{n+1}$};
    \draw[->,thick] (wp) to[bend right=38] node[midway,above] {$Y$} (vrp);
    \draw[->,thick] (vrp)--(tp) node[midway,below] {$R_n$};
    \draw[->,thick] (tp)--(wp) node[midway,below] {$D$};
    \draw[->,thick] (wp)--(tn1p) node[midway,above] {$R_{n+1}$};
    \node at (.8,-2.15) {$\ell>0$};
  \end{scope}
\end{tikzpicture}
\caption{The mixed complex after lifting the inductive GS sequence.  In
both panels the old inner cycle contains
$v_R\xrightarrow{X}z\xrightarrow{L_{n+1}}w\xrightarrow{Y}v_R$.  When
$\ell=0$, the edge $R_n\colon v_R\to w=t$ is a chord.  When $\ell>0$,
the corresponding chord is subdivided as
$v_R\xrightarrow{R_n}t\xrightarrow{D}w$.  In both cases
$R_{n+1}\colon w\to\ta_{n+1}$.}
\label{fig:mixed-complex}
\end{figure}

\emph{Normalizing the lifted cells.}
The partition gap path of $b_{n+1}$ factors as
\[
\gap{b_{n+1}}=P\,R_n\,D,
\]
where $P$ is the subpath from $L_{n+1}$ to $R_n$, viewed in
$K_{\mathrm{tr}}$ as a path from $w$ to $v_R$.  Hence
\[
\lambda_{n+1}^*\colon
L_{n+1}=L_{n+1}\phi_\Sigma(P)R_nD,
\]
\[
\rho_{n+1}^*\colon
R_{n+1}=\phi_\Sigma(P)R_nDR_{n+1}.
\]
The path $\phi_\Sigma(P)$ runs in the smaller admissible core from $w$
to the inner vertex $v_R$.  It cannot contain $R_n$, since any path in
that core containing $R_n$ ends at its terminal vertex; and it cannot
contain $A_1$, since no edge enters $\io$.  Thus it uses only the inner
cycle.  Write
\[
\phi_\Sigma(P)=Y C(v_R)^s
\]
for some $s\geq0$.
Using the cell $R_n=C(v_R)R_n$, apply (M2) repeatedly to absorb the $s$
full cycles.  The lifted cells become
\[
\lambda'\colon L_{n+1}=L_{n+1}YR_nD,
\qquad
\rho'\colon R_{n+1}=YR_nDR_{n+1}.
\]

\emph{The geometric substitution.}
The cell of $R_n$ is
\begin{equation}\label{eq:right-factor}
R_n=XL_{n+1}YR_n.
\end{equation}
Adjoin a fresh edge
\[
p\colon z\longrightarrow t
\]
with positive cell $p=L_{n+1}YR_n$.  Replacing this occurrence by $p$ in
$\lambda'$ and in \eqref{eq:right-factor} gives
\[
\lambda''\colon L_{n+1}=pD,
\]
\begin{equation}\label{eq:right-short}
R_n=Xp.
\end{equation}
The endpoints used below are summarized by
\[
\begin{array}{c|c}
D&t\longrightarrow w\\
Y&w\longrightarrow v_R\\
X&v_R\longrightarrow z\\
p&z\longrightarrow t.
\end{array}
\]

Use (M2) to substitute $pD$ for every remaining occurrence of
$L_{n+1}$, and then delete $L_{n+1}$ and $\lambda''$ by
(M1)$^{-1}$.  Next substitute $Xp$ for every remaining occurrence of
$R_n$, using \eqref{eq:right-short}, and delete $R_n$ and that cell.
Call the resulting complex $K_{\mathrm{fin}}$.  Its positive cells are
\[
e=e\,C(\ta(e))[L_{n+1}:=pD]
\quad\text{for each old active inner edge $e$ and for $e=A_1$},
\]
\[
p=pDYXp,
\qquad
R_{n+1}=YXpD\,R_{n+1}.
\]

\emph{Recognition of the final core.}
The inner edges of $K_{\mathrm{fin}}$ are the $n-1$ old active inner
edges, the fresh edge $p$, and the $k$ dummy edges of the original
partition.  There are $n+k$ of them.  The number of inner vertices is
also $n+k$: $K_{\mathrm{tr}}$ has $n+k-\ell$ inner vertices; for
$\ell>0$ the larger complex adds $\ell-1$ inner vertices of $D$ and turns the old terminal
vertex into an inner vertex, while for $\ell=0$ it identifies the old
terminal vertex with $w$ and then makes that class inner.  By
\cref{lem:skeleton-invariants,lem:cycle-rec}, these edges form a directed
simple cycle through all inner vertices.

For an old active edge $e$, and likewise for $A_1$, its old full-cycle
path traversed every old inner edge once, including $L_{n+1}$.  Replacing
that occurrence by $pD$ produces a closed path traversing every new
inner edge once, so it is the full new cycle based at $\ta(e)$.  The word
$DYXp$ is closed at $\ta(p)=t$ by the endpoint table and
traverses every new inner edge once; hence the cell of $p$ is
$p=pC(\ta(p))$.  Finally $YXpD$ is the cyclic rotation of $DYXp$ based at
$w=\io(R_{n+1})$, so the last cell is
\[
R_{n+1}=C(w)R_{n+1}.
\]
Therefore $K_{\mathrm{fin}}$ is an admissible core for
$(B,A_1\cdots A_N)$.

The lifted sequence never deleted $A_1$, $R_n$, or a dummy edge of the
truncated partition.  The additional moves delete only $L_{n+1}$ and
$R_n$, both active edges of the original complex $K_{\mathrm{full}}$.
Thus $A_1$, $A_N=R_{n+1}$, and every original dummy edge survives, completing
the induction.
\end{proof}

\section{Proof of the Main Theorem}\label{sec:mainproof}

\begin{theorem}\label{thm:main-formal}
For $n\geq2$, every group generated by an irreducible geometrically fast
set of $n$ positive bumps is isomorphic to $F_n$.
\end{theorem}

\begin{proof}
Let $G=\langle B\rangle$, where $B$ is irreducible and fast and
$|B|=n\geq2$.  By \cref{thm:peelable,cor:swap-groups}, there is a fast
set $B'$ with the same generated group whose dynamical diagram $D'$ is
peelable.  By \cref{rem:canonical-exists}, choose a fast realization
$B''$ of $D'$ in canonical configuration.  By
\cref{thm:diagram-isomorphism},
$\langle B''\rangle\cong\langle B'\rangle=G$.

The diagram $D'$ is irreducible with at least two bumps, so it has no
isolated bumps in the sense of \cref{def:isolated}.  Its canonical
partition therefore consists of its $2n$ feet,
$A_1,\dots,A_{2n}$, and has no dummy intervals.  By
\cref{thm:representation},
\[
G\cong\DG\bigl(\Kdyn(B''),A_1A_2\cdots A_{2n}\bigr).
\]
The Reduction Theorem gives a GS sequence $\Sigma$ from $\Kdyn(B'')$ to
an admissible core $K$, here a shifted core of $F_n$.  Transporting the
base path yields
\[
\DG\bigl(\Kdyn(B''),A_1\cdots A_{2n}\bigr)
\cong
\DG\bigl(K,\phi_\Sigma(A_1\cdots A_{2n})\bigr).
\]
The transported word is a $1$-path from $\io$ to $\ta$.  By
\cref{cor:dilated-Fn}, the latter diagram group is isomorphic to $F_n$.
Thus $G\cong F_n$.
\end{proof}

Consequently, $\mathcal C_n$ consists of a single isomorphism class,
represented by $F_n$.

\begin{remark}[Further applications]\label{rem:future}
The techniques of this paper apply beyond the class $\mathcal{C}_n$.
In fact, they can also be used to give alternative proofs to the
isomorphism results from \cite[Sections~5--7]{GolanAllWay}.
More generally, similar techniques show that a large class of closed
subgroups of Thompson's groups $F_m$ are isomorphic to Thompson
groups and enable the construction, for all $n\geq m\geq2$, of a
maximal subgroup of Thompson's group $F_m$ isomorphic to $F_n$.  We
will expand on these results in future papers.
\end{remark}

\end{document}